\documentclass{ejm}
\usepackage{amsmath}
\usepackage{graphicx}
\usepackage{amssymb}
\usepackage{caption}
\usepackage{subcaption}
\title[Phase separation for spherical biomembranes]{Domain formation via phase separation  for spherical biomembranes with small deformations}
\author[C. M. Elliott and L. Hatcher]{%
C.\ns M.\ns E\ls L\ls L\ls I\ls O\ls T\ls T$\,^1$,\ns
\and	
L.\ns H\ls A\ls T\ls C\ls H\ls E\ls R$\,^1$\ns
}

\affiliation{%
$^1\,$Mathematics Institute, Zeeman Building, University of Warwick, Coventry. CV4 7AL. UK\\
email\textup{\nocorr: \texttt{C.M.Elliott@warwick.ac.uk, L.Hatcher@warwick.ac.uk}}}

\date{2019}
\pubyear{2000}
\volume{000}
\pagerange{\pageref{firstpage}--\pageref{lastpage}}

\newtheorem{theorem}{Theorem}[section]
\newtheorem{Proposition}[theorem]{Proposition}

\newtheorem{Problem}[theorem]{Problem}

\DeclareMathOperator{\Tr}{Tr}
\def\Xint#1{\mathchoice
{\XXint\displaystyle\textstyle{#1}}%
{\XXint\textstyle\scriptstyle{#1}}%
{\XXint\scriptstyle\scriptscriptstyle{#1}}%
{\XXint\scriptscriptstyle\scriptscriptstyle{#1}}%
\!\int}
\def\XXint#1#2#3{{\setbox0=\hbox{$#1{#2#3}{\int}$ }
	\vcenter{\hbox{$#2#3$ }}\kern-.6\wd0}}

\def\dashint{\Xint-}
\numberwithin{equation}{section}

\newcommand{\spon}{H_s}	

\newcommand{\mc}{H}
\newcommand{\gc}{K}
\newcommand{\gbr}{\kappa_G}

\newcommand{\mbr}{\kappa}
\newcommand{\st}{\sigma}
\newcommand{\lt}{b}
\newcommand{\hf}{u}
\newcommand{\op}{\phi}
\newcommand{\cp}{\Lambda}
\DeclareMathOperator{\T}{T_\mathnormal{x}\Gamma}

\begin{document}
\label{firstpage}
\maketitle
\begin{abstract}
We derive and analyse an energy to model lipid raft formation on biological membranes involving a coupling between the local 
mean curvature  and the local composition. We apply a perturbation method recently introduced by Fritz, Hobbs and the 
first author to describe the geometry of the surface  as a 
graph over an undeformed Helfrich energy minimising surface. The result is a surface Cahn-Hilliard functional coupled with  a 
small deformation energy  We show that suitable minimisers 
of this energy exist and consider a gradient flow with conserved Allen-Cahn dynamics, for which existence and uniqueness 
results are proven. Finally, numerical simulations show that for the long time behaviour raft-like structures can  emerge and 
stablise, and their parameter dependence is further explored.
\end{abstract}

\section{Introduction}
Biological membranes are permeable barriers which separate cells from their exterior, and consist of various molecules such as proteins embedded within a lipid bilayer structure. They are of partiular mathematical interest since they can exhibit a variety of shape transition behaviour such as bud-formation or vessicle fission and fusion \cite{mcmahon2005membrane}. Following the pioneering works of Canham and Helfrich \cite{canham1970minimum,helfrich1973elastic}, the established continuum model treats the biomembrane as an infinitesimally thin deformable surface. The associated elastic bending energy (the so called Canham-Helfrich energy), accounting for possible surface tension is given by,
\begin{equation}
\label{eqn-helfrich}
\mathcal{E}(\Gamma):=\int_{\Gamma}\left(\frac{1}{2}\mbr(\mc-\spon)^2+\st+\gbr\gc\right){\rm d}\Gamma.
\end{equation}
Here $\Gamma=\partial\Omega$ is a two-dimensional hypersurface in $\mathbb{R}^3$ modelling the biomembrane and is given by the boundary of an open, bounded, connected set $\Omega\subset \mathbb{R}^3$. The parameters $\mbr>0$ and $\gbr>0$ are bending rigidities, $\spon$ is called the spontaneous curvature which is a measure of stress within the membrane for the flat configuration, $\mc$ is the mean curvature, $\gc$ is the Gauss curvature and $\st\geq 0$ is the surface tension.

Biomembranes consisting of multiple differing lipid types can undergo phase separation, forming a disordered phase where the lipid molecules can diffuse more freely and an ordered phase where the lipid molecules are more tightly packed together. A connected field of study with large academic interest (for example see \cite{bassereau2018physics}) involves the nature of membrane rafts, more commonly referred to as lipid rafts which were first introduced in \cite{simons1997functional}. These are small (10-200nm), relatively ordered domains which are enriched with cholestorol and sphingolipids and are understood to compartmentalise cellular processes such as signal transduction, protein sorting and are important for other mechanisms such as host-pathogen interactions \cite{pike2006rafts}. 

Since the size of these rafts are beyond the diffraction limit, direct microscopic observation has not been possible. Experimental resuls have been limited to observations on larger artificial membranes whose composition lack the complexity of biomembranes, or using alternative microscopy techniques such as fluorescence microscopy which alters the composition of the membrane. In both cases the \emph{in vivo} inferences drawn are questionable and the field has remained controversial \cite{sezgin2017mystery}. 
Since the dynamics and processes governing the formation and maintenance of lipid rafts are not well understood a number of explanations have been offered. One suggestion is that raft formation is driven by cholestorol pinning, and a model for this was recently proposed by Garcke et al. \cite{AbeKam19,garcke2016coupled}. In this paper  we consider whether the membrane geometry is a sufficient mechanism driving the formation of these rafts via protein interactions.

Experimental results on artificial membranes have shown there exists a correlation between the composition of the different phases and the local membrane curvature \cite{baumgart2003imaging,rinaldin2018geometric,parthasarathy2006curvature,hess2007shape}. Proteins are both able to sense whether the local environment matches their curvature preference, as well as in large enough numbers induce that curvature upon the membrane \cite{callan2013curvature}. Since proteins have a preference for raft type regions we consider here whether phase dependent material parameters offers a possible explanation for the domain formation observed. To that end we introduce an order parameter $\phi$, and consider the energy
\begin{equation}
\label{eqn-helfrich+tension}
\mathcal{E}(\Gamma):=\int_{\Gamma}\left(\frac{1}{2}\mbr(\phi)(\mc-\spon(\phi))^2+\st+\gbr(\phi)\gc+\frac{\lt\epsilon}{2}|\nabla_\Gamma\phi|^2+\frac{\lt}{\epsilon}W(\phi)\right){\rm d}\Gamma.
\end{equation}

The energy \eqref{eqn-helfrich+tension} is a modified version of \eqref{eqn-helfrich} where we have included a Ginzburg-Landau energy functional with  coefficient $\lt>0$, to incorporate the line tension between the two phases as well as making explicit the dependence of the bending rigidities and spontanerous curvature on the phasefield. Here $W(\op)$ is a smooth double well potential, with the local minimisers corresponding to the value $\phi$ takes in the respective phases, and $\epsilon>0$ is a small parameter commensurate with the width of the interface. 

An energy of this type was first proposed by Leibler \cite{leibler1986curvature}. In that case, the only material property taken to be dependent on the phase field was the spontaneous curvature, which was assumed to take the form
\begin{equation}
\label{eqn-leibler}
\spon(\phi)=\Lambda\phi,
\end{equation}
where $\Lambda\in\mathbb{R}$ is the curvature coefficient. 

An energy of the general form given in \eqref{eqn-helfrich+tension} was considered in \cite{elliott2010modeling,elliott2010surface,elliott2013computation} from  computational and formal asymptotics perspectives. The associated variational  problem is highly nonlinear and leads to a free boundary on a free boundary.  Other models have been suggested, such as in \cite{healey2017symmetry,wang2008modelling}. Here we utilise a recent pertubation approach for approximately spherical biomembranes introduced in \cite{elliott2016small}, in order to simplify \eqref{eqn-helfrich+tension}. This approach for flat domains using the Monge gauge approximation was considered in \cite{elliott2016variational}. The result is a variational problem on a fixed spherical domain.

In order to apply the above mentioned perturbation approach we make the following additional assumptions on \eqref{eqn-helfrich+tension}: the only material parameter that depends on the phase field is the spontaneous curvature, which we take to have the form given in \eqref{eqn-leibler}; we rescale the coefficient $\cp$ and replace by $\rho\cp$, and rescale $\lt$ and replace by $\rho^2\lt$; the volume of $\Gamma$ is fixed, as well as the integral of the phasefield. The justification for these assumptions is as follows: a spontanous curvature of this type corresponds to the simple assumption that the proteins induce a curvature proportional to their area concentration; the $\rho$ scaling of the spontaneous curvature induces order $\rho$ deformations of the surface,  the $\rho^2$ scaling is motivated since the line tension for lipid rafts has been calculated to depend quadratically on the spontaneous curvature \cite{kuzmin2005line}; the volume contraint corresponds to the impermeability of the membrane, and the order parameter constraint corresponds to a conservation of mass law on the embedded membrane proteins. After making these assumptions we obtain the following energy from \eqref{eqn-helfrich+tension}
\begin{equation}
\label{eqn-helfrich+tension+assumptions}
\mathcal{E}(\Gamma):=\int_{\Gamma}\left(\frac{1}{2}\mbr(\mc-\rho\Lambda\phi)^2+\st+\gbr\gc+\frac{\rho^2\lt\epsilon}{2}|\nabla_\Gamma\phi|^2+\frac{\rho^2\lt}{\epsilon}W(\phi)\right){\rm d}\Gamma,
\end{equation}
subject to a volume constraint and mean value constraint.

	We remark that in the case that $\Gamma$ is a closed hypersurface (without boundary), then the Gauss-Bonnet Theorem gives that
	\begin{equation}
	\int_{\Gamma}K=2\pi\chi(\Gamma),
	\end{equation}
	where $\chi(\Gamma)$ is the Euler characteristic of $\Gamma$. So in the case the material parameter $\gbr$ is independent of the phase field, then the Gauss curvature term can be dropped from \eqref{eqn-helfrich+tension+assumptions}.

The rest of the paper is set out as follows. In section 2 we briefly cover the notation and some preliminaries on surface calculus. In Section 3 we give the details of the perturbation approach alluded to above and derive an energy  that approximates \eqref{eqn-helfrich+tension}. In section 4 we prove that within a suitable space minimisers exist to this approximate energy. In section 5 we consider a gradient flow and prove existence and uniquness results for these equations, before finally in section 6 we conduct some numerical experiments.
\section{Notation and preliminaries}
Within this section we state the basic definitions and results for a two dimensional $C^2-$hypersurface $\Gamma$ which will be used throughout this paper. For a thorough treatment of the material covered here we refer the reader to \cite{dziuk2013finite}.

Given a  point $x\in\Gamma$, with unit normal $\nu$, an open subset $U$ of $\mathbb{R}^{3}$ containing $x$, and a function $u\in C^1(U)$ we define the tangential gradient of $u,\nabla_\Gamma u$ by
\begin{align}
\nabla_\Gamma u=\nabla u-(\nabla u\cdot \nu)\nu,
\end{align}
and denote it's components by
\begin{align}
\nabla_\Gamma u=(\underline{D}_1u,\underline{D}_2u,\underline{D}_3u).
\end{align}

We can also define the Laplace-Beltrami operator of $u$ at $x$  by
\begin{align}
\Delta_\Gamma u(x)=\sum_{i=1}^{3}\underline{D}_i\underline{D}_iu(x).
\end{align}

Denoting the tangent space of $\Gamma$ at $x$ by $\T$, we define the Weingarten map $\mathcal{H}:\T\to\T$ by
$\mathcal{H}:=\nabla_\Gamma \nu$  with eigenvalues given by the principle curvatures $\kappa_1$ and $\kappa_2$. The mean curvature is then given by
\begin{align}
H:=\Tr(\mathcal{H})=\kappa_1+\kappa_2,
\end{align}
which differs from the normal definition by a factor of 2. The Gauss curvature is then given by
\begin{align}
K:=\det(\mathcal{H})=\kappa_1\kappa_2.
\end{align}
We can consider the extended Weingarten map $\mathcal{H}:\mathbb{R}^3\to\T$ by defining $\mathcal{H}$ to have zero eigenvalue in the normal direction.

For $p\in[1,\infty)$ we define $L^p(\Gamma)$ to be the space of functions $u:L^p(\Gamma)\to\mathbb{R}$ which are measurable with respect to the surface measure $\mathrm{d}\Gamma$ and have finite norm
\begin{align}
\|u\|_{L^p(\Gamma)}=\left(\int_{\Gamma}^{}|u|^p \:\mathrm{d}\Gamma\right)^\frac{1}{p}.
\end{align}

We say a function $u\in L^1(\Gamma)$ has the weak derivative $v_i=\underline{D}_iu$, if for every function $\phi\in C^1(\Gamma)$ with compact support $\overline{\{x\in\Gamma:\phi(x)\neq 0\}}$$\subset\Gamma$ we have the relation
\begin{align}
\int_{\Gamma}^{}u\underline{D}_i\phi \:\mathrm{d}\Gamma=-\int_{\Gamma}^{}\phi v_i \:\mathrm{d}\Gamma+\int_{\Gamma}^{}u\phi H\nu_i\:\mathrm{d}\Gamma.
\end{align}

We define the Hilbert spaces $H^1(\Gamma)$ and $H^2(\Gamma)$ by
\begin{align}
H^1(\Gamma):&=\left\{f\in L^2(\Gamma):f\text{ has weak derivatives }\underline{D}_if\in L^2(\Gamma), i\in\{1,2,3\}\right\},
\\
H^2(\Gamma):&=\left\{f\in H^1(\Gamma):f\text{ has weak derivatives }\underline{D}_i\underline{D}_jf\in L^2(\Gamma), i,j\in\{1,2,3\}\right\},
\end{align}
\
with inner products given by
\begin{align}
(u,v)_{H^1(\Gamma)}&=\int_{\Gamma}\left(\nabla_\Gamma\cdot\nabla_\Gamma v+uv\right) \:\mathrm{d}\Gamma,
\\
(u,v)_{H^2(\Gamma)}&=\int_{\Gamma}\left(\Delta_\Gamma u\Delta_\Gamma v+\nabla_\Gamma\cdot\nabla_\Gamma v+uv\right) \:\mathrm{d}\Gamma.
\end{align}
We comment that the inner products given above are not the standard ones used, but the induced norms are equivalent to the usual norms in the case $\Gamma$ is a closed surface, see \cite{dziuk2013finite}. 

Integration by parts is then given by
\begin{theorem}
\label{By Parts}
Let $\Gamma$ be a bounded $C^2-$hypersurface (without boundary) and suppose $u\in H^1(\Gamma)$ and $v\in H^2(\Gamma)$. Then
\begin{align}
\int_{\Gamma}\nabla_\Gamma u\cdot\nabla_\Gamma v \:\:{\rm d}\Gamma=-\int_{\Gamma}u\Delta_\Gamma v\:\: {\rm d}\Gamma.
\end{align}
\end{theorem}
Finally, given a family of evolving hypersurfaces $(\Gamma(t))_{t\in[0,T]}$ and velocity $v:\mathcal{G}\to\mathbb{R}^3$ where $\mathcal{G}=\cup_{t\in[0,T]}(\Gamma(t)\times\{t\})$ we consider $(x_0,t_0)\in\mathcal{G}$ and denote by $\gamma:(t_0-\delta,t_0+\delta)\to\mathbb{R}^3$ the unique solution to the initial value problem
\begin{align}
\gamma^\prime(t)=v(\gamma(t),t),\quad\gamma(t_0)=x_0.
\end{align}  
Then for a function $f:\mathcal{G}\to\mathbb{R}$ we define the material time derivative by
\begin{align}
\partial^\bullet_t f(x_0,t_0):=\left.\frac{d}{dt}f(\gamma(t),t)\right|_{t=t_0}.
\end{align}
The transport theorem is then given by
\begin{theorem}[Transport Theorem]
	\label{Transport}
	Let $\Gamma(t)$ be an evolving surface with velocity field $v$. Then assuming that $f$ is a function such that all the following quantities exist, then
	\begin{align}
	\frac{d}{dt}\int_{\Gamma(t)}f\:{\rm d}\Gamma(t)=\int_{\Gamma(t)}\partial^\bullet_t f+f\nabla_\Gamma\cdot v\:{\rm d}\Gamma.
	\end{align}
\end{theorem}
\section{Derivation of Model}
In this section we apply the perturbation approach detailed in \cite{elliott2016small} in order to obtain an approximate energy to \eqref{eqn-helfrich+tension}. We first consider the Lagrangian 
\begin{align}
\label{eqn-Lagrangian}
\mathcal{L}(\Gamma,\lambda):=\mbr\mathcal{W}(\Gamma)+\sigma\mathcal{A}(\Gamma)+\lambda(\mathcal{V}(\Gamma)-V_0).
\end{align}
where
\begin{align}
\mathcal{W}(\Gamma)&=\int_{\Gamma}\frac{1}{2}H^2\:{\rm d}\Gamma,&
\mathcal{A}(\Gamma)&=\int_{\Gamma}1\:{\rm d}\Gamma,&
\mathcal{V}(\Gamma)&=\int_{\Gamma}\frac{1}{3}\text{Id}_\Gamma\cdot\nu\:{\rm d}\Gamma.
\end{align}
 Here $\mathcal{W}$ denotes the Willmore energy, $\mathcal{A}$ the area functional and $\mathcal{V}$  the enclosed volume.  Since the Willmore energy is scale invariant and  the area is not, the volume is constrained using a Lagrange multiplier $\lambda$ and with fixed volume $V_0$. In addition $\mathcal{A}(\Gamma)$ and $\mathcal{V}(\Gamma)$ must satisfy the isoperimetric inequality
 \begin{equation}
 \mathcal{A}^3(\Gamma)\geq 36\pi \mathcal{V}^2(\Gamma).
 \end{equation}

In \cite{elliott2016small} it was shown that \eqref{eqn-Lagrangian} has a critical point $(\Gamma_0,\lambda_0)$, where $\Gamma_0=S(0,R)$, the sphere of radius $R$ centred at the origin and $\lambda_0=-\frac{2\sigma}{R}$. Applying a small forcing term $\rho\mathcal{F}$ we expect a critical point of the perturbed Lagrangian
\begin{align}
\label{eqn-lagrangian}
\mathcal{L}_\rho(\Gamma,\phi,\lambda,\mu):=&\mbr\mathcal{W}(\Gamma)+\sigma\mathcal{A}(\Gamma)+\lambda(\mathcal{V}(\Gamma)-V_0)
+\rho\mathcal{F}(\Gamma,\phi,\mu).
\end{align}
to be of the form $(\Gamma_\rho,\phi_\rho,\lambda_\rho,\mu_\rho)$ where $\Gamma_\rho$ and $\lambda_\rho$ are perturbations given by
\begin{align}
\Gamma_\rho&=\{p+\rho(\hf\nu_0)(p):p\in\Gamma_0\},\\
\lambda_\rho&=\lambda_0+\rho\lambda_1,
\label{eqn-deform}
\end{align}
of the critical point $(\Gamma_0,\lambda_0)$ for the non-perturbed Lagrangian $\mathcal{L}$.
Here $\nu_0$ is the unit normal to $\Gamma_0$, $\rho\in\mathbb{R}$ such that $\rho\ll 1$ and $\hf\in C^2(\Gamma,\mathbb{R})$ is the height function that describes the deformation.

Since $(\Gamma_\rho,\phi_\rho,\lambda_\rho,\mu_\rho)$ is a critical point it follows that
\begin{align}
\label{Lagrangian2}
\begin{cases}
\left.\frac{d}{ds}\mathcal{L}_\rho(\Gamma_\rho,\phi_\rho,\lambda_\rho,\mu_\rho+s\zeta)\right|_{s=0}=0&\qquad\forall\zeta\in\mathbb{R},
\\
\left.\frac{d}{ds}\mathcal{L}_\rho(\Gamma_\rho,\phi_\rho,\lambda_\rho+s\xi,\mu_\rho)\right|_{s=0}=0&\qquad\forall\xi\in\mathbb{R},
\\
\left.\frac{d}{ds}\mathcal{L}_\rho(\Gamma_\rho,\phi_\rho+s\eta,\lambda_\rho,\mu_\rho)\right|_{s=0}=0&\qquad\forall\eta\in C^1(\Gamma_\rho),
\\
\left.\frac{d}{ds}\mathcal{L}_\rho(\Gamma_\rho^s,\phi_\rho,\lambda_\rho,\mu_\rho)\right|_{s=0}=0&\qquad\forall g\in C^2(\Gamma_\rho),
\end{cases}
\end{align}
where $\Gamma_\rho^s:=\left\{x+sg\nu_\rho(x):x\in\Gamma_\rho\right\}$ and $\nu_\rho$ is the unit normal to $\Gamma_\rho$.

We apply this perturbation method for the case that the forcing term $\mathcal{F}$ is given by
\begin{align}
\mathcal{F}(\Gamma,\phi,\mu)=\mathcal{F}_1(\Gamma,\phi)+\rho\mathcal{F}_2(\Gamma,\phi)+\mu(\mathcal{C}(\Gamma,\phi)-\alpha),
\end{align}
where
\begin{align}
\mathcal{F}_1(\Gamma,\phi)&:=-\int_{\Gamma}H\phi \:{\rm d}\Gamma,&
\mathcal{F}_2(\Gamma,\phi)&:=\int_{\Gamma}\lt\left(\frac{\epsilon}{2}|\nabla_\Gamma\op|^2+\frac{1}{\epsilon} W(\phi)+\frac{\mbr\cp^2\op^2}{2\lt}\right)\:{\rm d}\Gamma,
\end{align}
are two forcing terms obtained from \eqref{eqn-helfrich+tension+assumptions} and $\mu$ is a Lagrange multiplier for the mean value constraint functional
\begin{align}
\mathcal{C}(\Gamma,\phi)&:=\dashint_{\Gamma}\phi \:{\rm d}\Gamma=\alpha.
\end{align}

Since we are interested in doing a Taylor approximation of \eqref{eqn-lagrangian}, we need to calculate the first and second variations of some of the energy functionals above. We remark that in our case when determing the second variation it is sufficient to find the first variation of the first variation, although in general this need not be the case, see Remark 3.2 in \cite{elliott2016small}. 


We first state the following results, proofs of which can be found in the appendix of \cite{elliott2016small}.

\begin{align}
\mathcal{W}^\prime(\Gamma_0)[u\nu_0]&=0,
&
\mathcal{W}^{\prime\prime}(\Gamma_0)[u\nu_0,u\nu_0]&=\int_{\Gamma_0}\left((\Delta_{\Gamma_0}\hf)^2-\frac{2}{R^2}|\nabla_{\Gamma_0} \hf|^2\right)\:{\rm d}\Gamma_0,
\\
\mathcal{V}^\prime(\Gamma_0)[u\nu_0]&=\int_{\Gamma_0} \hf\:{\rm d}\Gamma_0,&
\mathcal{V}^{\prime\prime}(\Gamma_0)[u\nu_0,u\nu_0]&=\int_{\Gamma_0} H_0\hf^2\:{\rm d}\Gamma_0,
\\
\mathcal{A}^\prime(\Gamma_0)[u\nu_0]&=\int_{\Gamma_0} H_0\hf\:{\rm d}\Gamma_0,&
\mathcal{A}^{\prime\prime}(\Gamma_0)[u\nu_0,u\nu_0]&=\int_{\Gamma_0}\left( |\nabla_{\Gamma_0}\hf|^2+\frac{2\hf^2}{R^2}\right)\:{\rm d}\Gamma_0,
\end{align}
where we have denoted the mean curvature on $\Gamma_0$ and $\Gamma_\rho$ by $H_0$ and $H_\rho$ respectively. Similarly we will denote the extended Weingarten map on $\Gamma_0$ and $\Gamma_\rho$ by $\mathcal{H}_0$ and $\mathcal{H}_\rho$. For ease of notation we will also write $\tau_0=\left.\tau_\rho\right|_{\rho=0}$ and $\tau_1=\left.\partial^\bullet_\rho\tau_\rho\right|_{\rho=0}$ where $\tau$ is a placeholder for $\phi$ and $\mu$. 

It will be sufficient for our purposes to additionally only calculate the first variation of $\mathcal{F}(\Gamma,\phi,\mu)$,
\begin{align}
\begin{split}
\mathcal{F}^\prime(\Gamma_0,\phi_0,\mu_0)[u\nu,\phi_1,\mu_1]=&\mathcal{F}_1^\prime(\Gamma_0,\phi_0)[u\nu,\phi_1]+\mathcal{F}_2(\Gamma_0,\phi_0)
\\
&+\mu_1\left(\mathcal{C}(\Gamma_0,\phi_0)-\alpha\right)+\mu_0\mathcal{C}^\prime(\Gamma_0,\phi_0)[u\nu,\phi_1],
\end{split}
\end{align} 
which amounts to calculating the first variation of $\mathcal{F}_1(\Gamma,\phi)$ and $\mathcal{C}(\Gamma,\phi)$.	By applying Theorem \ref{Transport} and using that $\partial^\bullet_\rho H_\rho=-\Delta_{\Gamma_\rho} u-|\mathcal{H}_\rho|^2 u$ (see Corollary A.1 in \cite{elliott2016small}) we obtain
	\begin{align}
	\begin{split}
	\left.\frac{d}{d\rho}\mathcal{F}_1(\Gamma_\rho,\phi_\rho)\right|_{\rho=0}&=-\left.\int_{\Gamma_\rho}\partial^\bullet_\rho(H_\rho\phi_\rho)+H_\rho\phi_\rho\nabla_{\Gamma_\rho}\cdot(\hf\nu_\rho)\:{\rm d}\Gamma_\rho\right|_{\rho=0}
	\\
	&=\int_{\Gamma_0}\phi_0 \Delta_{\Gamma_0}\hf+\phi_0|\mathcal{H}_0|^2\hf-H_0\phi_1- H_0^2\phi_0\hf\:{\rm d}\Gamma_0,
	\end{split}
	\end{align}
	and hence using that $|\mathcal{H}_0|^2=\frac{2}{R^2}$ and $H_0=\frac{2}{R}$ gives that,
	\begin{align}
	\mathcal{F}_1^\prime(\Gamma_0,\phi_0)[u\nu_0,\phi_1]=\int_{\Gamma_0}\phi_0\left(\Delta_{\Gamma_0}\hf-\frac{2\hf}{R^2}\right)-\frac{2\phi_1}{R}\:{\rm d}\Gamma_0.
	\end{align}
	
	Similarly we obtain
	\begin{align}
	\begin{split}
	\mathcal{C}^\prime(\Gamma_0,\phi_0)[u\nu_0,\phi_1]&=\left.\dashint_{\Gamma_0}\phi_1+\phi_0\nabla_\Gamma\cdot(\hf\nu_0)\:{\rm d}\Gamma_0\right.-\frac{\left.\frac{d}{d\rho}\int_{\Gamma_\rho}1\:{\rm d}\Gamma_\rho\right|_{\rho=0}}{\int_{\Gamma_0}1\:{\rm d}\Gamma_0}\dashint_{\Gamma_0}\phi_0\:{\rm d}\Gamma_0
	\\
	&=\dashint_{\Gamma_0}\left(\phi_1+\frac{2\phi_0u}{R}\right)\:{\rm d}\Gamma_0-\frac{2}{R}\left(\dashint_{\Gamma_0}\hf\:{\rm d}\Gamma_0\right)\left(\dashint_{\Gamma_0}\phi_0\:{\rm d}\Gamma_0\right).
	\end{split}
	\end{align}
	We can determine $\mu_0$ explicitly since from \eqref{Lagrangian2} we have that
	\begin{align}
	\frac{d}{ds}\rho\mathcal{F}(\Gamma_\rho,\phi_\rho+s\eta,\mu_\rho)=0,
	\end{align}
	and therefore
	\begin{align}
	\mathcal{F}_1(\Gamma_0,\eta)+\mu_0\mathcal{C}(\Gamma_0,\eta)=0,
	\end{align}
	from which we obtain that $\mu_0=\frac{2|\Gamma_0|}{R}$.
	It therefore follows that
	\begin{align}
	\nonumber
	\mathcal{F}^\prime(\Gamma_0,\phi_0,\mu_0)[u\nu,\phi_1,\mu_1]=&\int_{\Gamma_0}\left[\phi_0\Delta_{\Gamma_0}u+\frac{2\phi_0u}{R^2}+\frac{b\epsilon}{2}|\nabla_{\Gamma_0}\phi_0|^2+\frac{b}{\epsilon}W(\phi_0)+\frac{\kappa\Lambda^2\phi_0^2}{2}\right]\:{\rm d}\Gamma_0,
	\end{align}
	where above we have also used the linearised Lagrange multiplier constraints
	\begin{align}
	\int_{\Gamma_0}u\:{\rm d}\Gamma_0&=0&\dashint_{\Gamma_0}\phi_0\:{\rm d}\Gamma_0&=\alpha
	\end{align}
	which are obtained from \eqref{Lagrangian2}.
	 
We can now prove the following result.
\begin{theorem}
	\label{Taylor}
	With the assumptions given above it follows that
	\begin{align}
	\mathcal{L}_\rho(\Gamma_\rho,\phi_\rho,\lambda_\rho,\mu_\rho)=C_1+\rho C_2+\rho^2\mathcal{E}(\phi_0,\hf)+\mathcal{O}(\rho^3),
	\end{align}
	where
	\begin{align}
	\label{eqn-peturb-energy}
	\begin{split}
	\mathcal{E}(\phi_0,\hf):=&\int_{\Gamma_0}\frac{\kappa}{2}(\Delta_{\Gamma_0}\hf)^2+\frac{1}{2}\left(\sigma-\frac{2\kappa}{R^2}\right)|\nabla_{\Gamma_0}\hf|^2-\frac{\sigma \hf^2}{R^2}
	\\
	&\quad+\kappa\Lambda\phi_0\Delta_{\Gamma_0}\hf+\frac{2\kappa\Lambda \hf\phi_0}{R^2}+\frac{b\epsilon}{2}|\nabla_{\Gamma_0}\phi_0|^2+
	\frac{b}{\epsilon}W(\phi_0)+\frac{\kappa\Lambda^2\phi_0^2}{2}\:{\rm d}\Gamma_0
	\end{split}
	\end{align}
	for $C_1$ and $C_2$ constant.
\end{theorem}
\begin{proof}
	We wish to apply Taylor's Theorem so that we can obtain a good approximation to the perturbed Langrangian $\mathcal{L}_\rho(\Gamma_\rho,\phi_\rho,\lambda_\rho,\mu_\rho)$. Performing a second order Taylor expansion in $\rho$ we obtain that
	\begin{align}
	\begin{split}
	\mathcal{L}_\rho(\Gamma_\rho,\phi_\rho,\lambda_\rho,\mu_\rho)=&\mathcal{L}_0(\Gamma_0,\phi_0,\lambda_0,\mu_0)+\rho\left.\frac{d}{d\rho}\mathcal{L}_\rho(\Gamma_\rho,\phi_\rho,\lambda_\rho,\mu_\rho)\right|_{\rho=0}
	\\
	&+\frac{\rho^2}{2}\left.\frac{d^2}{d\rho^2}\mathcal{L}_\rho(\Gamma_\rho,\phi_\rho,\lambda_\rho,\mu_\rho)\right|_{\rho=0}+\mathcal{O}(\rho^3).
	\end{split}
	\end{align}
	We first observe that $\mathcal{L}_0(\Gamma_0,\phi_0,\lambda_0,\mu_0)=\kappa\mathcal{W}(\Gamma_0)+\sigma\mathcal{A}(\Gamma_0)$. For the second term we use that $(\Gamma_0,\lambda_0)$ is a critical point of $\mathcal{L}$ and obtain that
	\begin{align}
	\left.\frac{d}{d\rho}\mathcal{L}_\rho(\Gamma_\rho,\phi_\rho,\lambda_\rho,\mu_\rho)\right|_{\rho=0}&=\kappa\Lambda\mathcal{F}_1(\Gamma_0,\phi_0)=-\frac{2\kappa\Lambda}{R}\int_{\Gamma_0}\phi_0\:{\rm d}\Gamma_0=-8\kappa\Lambda\pi R\alpha.
	\end{align}
	We therefore see that the second order term is the lowest order term which depends on any of the variables. It remains to determine the form of this second order term. To do this we write
	\begin{align}
	\begin{split}
	\left.\frac{d^2}{d\rho^2}\mathcal{L}_\rho(\Gamma_\rho,\phi_\rho,\lambda_\rho,\mu_\rho)\right|_{\rho=0}=&\mbr\mathcal{W}^{\prime\prime}(\Gamma_0)[\hf\nu_0,\hf\nu_0]+\sigma\mathcal{A}^{\prime\prime}(\Gamma_0)[\hf\nu_0,\hf\nu_0]+\lambda_0\mathcal{V}^{\prime\prime}(\Gamma_0)[\hf\nu_0,\hf\nu_0]
	\\
	&+2\lambda_1\mathcal{V}^\prime(\Gamma_0)[\hf\nu_0]+2\mathcal{F}^\prime(\Gamma_0,\phi_0,\mu_0)[u\nu_0,\phi_1,\mu_1]
	\\
	=&2\mathcal{E}(\phi_0,u),
	\end{split}
	\end{align}
	as required.
\end{proof}
We note that formally taking $R\to\infty$ in \eqref{eqn-peturb-energy} we obtain the approximation given in \cite{leibler1986curvature} and more recently considered in \cite{fonseca2016domain} for a flat domain. It is this energy which we will study in the remainder of the paper. For ease of notation from now on we will denote $\Gamma_0$ by $\Gamma$ and $\phi_0$ by $\phi$.
\section{Energy minimisers}
We will restrict ourselves to considering the energy
$\mathcal{E}(\cdot,\cdot):\mathcal{K}\to \mathbb{R}$ given in \eqref{eqn-peturb-energy} for a $W:\mathbb{R}\to\mathbb{R}$ that satisfies the following properties:
\begin{enumerate}
\item $W(\cdot)\in C^1(\mathbb{R},\mathbb{R})$,
\item There exists $c_0\in \mathbb{R}^+$ such that $(W^\prime(r)-W^\prime(s))(r-s)\geq -c_0|r-s|^2$ $\forall r,s\in \mathbb{R}$,
\item There exists $c_1, c_2\in \mathbb{R}^+$ such that $c_1r^4-c_2\leq W(r)$, $\forall r\in\mathbb{R}$,
\item There exists $c_3, c_4\in \mathbb{R}^+$ such that $W^\prime(r)\leq c_3W(r)+c_4$,
\item There exists $c_5\in\mathbb{R}^+$ such that $W^\prime(r)r\geq-c_5r^2$,
\end{enumerate}
and for $\mathcal{K}$ given by
\begin{equation}
\mathcal{K}:=\left\{(\phi,\hf)\in H^1(\Gamma)\times H^2(\Gamma):\dashint_{\Gamma}\phi\:{\rm d}\Gamma=\alpha \text{ and }\hf\in \text{span}\{1,\nu_1,\nu_2,\nu_3\}^\perp\right\}.
\end{equation}
where the $\nu_i$ are the components of the normal $\nu$ of $\Gamma$ and orthogonality is understood in the $H^2(\Gamma)$ sense; although in this case it's equivalent to orthogonality in the $L^2(\Gamma)$ sense.
We motivate this choice of $\mathcal{K}$ as follows. The regularity required means a subspace of $H^1(\Gamma)\times H^2(\Gamma)$ is the natural choice to make. $\int_{\Gamma} u \:{\rm d}\Gamma=0$ is a linearised volume constraint which corresponds to membrane impermeability, $\dashint_\Gamma \phi \:{\rm d}\Gamma=\alpha$ is a linearised conservation of mass constraint on the membrane particles and $\int_{\Gamma}\hf\nu_i\:{\rm d}\Gamma=0$ for $i\in\{1,2,3\}$ are linearised translation invariance constraints on the membrane. Mathematically, these translation invariances arise since $\{\nu_1,\nu_2,\nu_3\}$ lie in the nullspace of $\mathcal{E}(\phi,\cdot)$.

We first address the question of existence.
\begin{Proposition}
\label{Prop-Direct}
There exists $(\phi^*,\hf^*)\in\mathcal{K}$ such that
\begin{displaymath}
\mathcal{E}(\phi^*,\hf^*)=\inf_{(\phi,\hf)\in\mathcal{K}}\mathcal{E}(\phi,\hf).
\end{displaymath}
\end{Proposition}
\begin{proof}
We have that $H^1(\Gamma)\times H^2(\Gamma)$ is a Hilbert space so it is reflexive and since $\mathcal{K}$ is a sequentially weakly closed subset of $H^1(\Gamma)\times H^2(\Gamma)$ then existence of a minimiser will follow from the Direct method (See Theorem 9.3-1 in \cite{ciarlet2013linear}) provided $\mathcal{E}(\cdot,\cdot):\mathcal{K}\to\mathbb{R}$ is coercive and sequentially weakly lower semicontinuous.

We note the  Poincare  type inequality given by
\begin{equation}
\label{eqn-poincare}
\int_{\Gamma}^{}u^2\:{\rm d}\Gamma\leq\frac{R^2}{6}\int_{\Gamma}^{}|\nabla_\Gamma u|^2\:{\rm d}\Gamma\leq\frac{R^4}{36}\int_{\Gamma}^{}(\Delta_\Gamma u)^2\:{\rm d}\Gamma,
\end{equation}
which holds for all  $\hf\in\text{span}\{1,\nu_1,\nu_2,\nu_3\}^\perp$ (see \cite{elliott2016small}). Using this, Young's inequality  and property (3) of $W(\cdot)$ it follows that there exists $C_1,C_2$ and $C_3\in\mathbb{R}^+$ such that,
\begin{equation}
\mathcal{E}(\phi,\hf)\geq C_1\|\hf\|^2_{H^2(\Gamma)}+ C_2\|\phi\|^2_{H^1(\Gamma)}-C_3.
\end{equation} 
Hence $\mathcal{E}(\cdot,\cdot):\mathcal{K}\to\mathbb{R}$ is coercive.

To prove that $\mathcal{E}(\cdot,\cdot):\mathcal{K}\to\mathbb{R}$ is sequentially weakly lower semi continuous we first note that the quadratic terms in $\hf$ form a bounded, symmetric and positive definite bilinear form and hence are weakly lower semi-continuous. A similar argument can be applied for the $|\nabla_\Gamma\phi|^2$ term. The remaining terms are also weakly lower semi-continuous by an application of a Rellich-Kondrachov type compactness embedding theorem \cite{aubin1982nonlinear}. This then completes the proof.

\end{proof}
\subsection{Euler-Lagrange equations}
\label{332}
Knowing that minimisers of \eqref{eqn-peturb-energy} exist, we want to say something about their structure. Therefore we compute the Euler equations associated with the energy functional $\mathcal{E}(\cdot,\cdot)$ over the space $\mathcal{K}$, and secondly  over the full space $H^1(\Gamma)\times H^2(\Gamma)$, by introducing the constraints as Lagrange multipliers. By applying Euler's Theorem (See Theorem 7.1-5 in \cite{ciarlet2013linear}) it follows that a critical point (and hence a minimiser $(\phi^*,\hf^*)$ of Proposition \ref{Prop-Direct}) is a solution of the following problem:
\begin{Problem}
\label{Prob-Euler}
Find $(\phi,\hf)\in \mathcal{K}$ such that
\begin{align}
\label{EL2}
\int_{\Gamma}\frac{b}{\epsilon}W^\prime(\phi)w+b\epsilon\nabla_\Gamma \phi\cdot\nabla_\Gamma w+\kappa\Lambda w\Delta_\Gamma \hf+\frac{2\kappa\Lambda \hf w}{R^2}+\kappa\Lambda^2\phi w\:{\rm d}\Gamma&=0,
\\
\label{EL1}
\int_{\Gamma}\kappa\Delta_\Gamma \hf\Delta_\Gamma v+\left(\sigma-\frac{2\kappa}{R^2}\right)\nabla_\Gamma \hf\cdot\nabla_\Gamma v-\frac{2\sigma}{R^2}\hf v+\kappa\Lambda\phi\Delta_\Gamma v+\frac{2\kappa\Lambda\phi v}{R^2}\:{\rm d}\Gamma&=0,
\end{align}
for all $w\in W:=\left\{\eta\in H^1(\Gamma):\int_{\Gamma}\eta\:{\rm d}\Gamma=0\right\}$ and for all $v\in V:=\{\eta\in H^2(\Gamma):\eta\in \text{span}\{1,\nu_1,\nu_2,\nu_3\}^\perp\}$. 
\end{Problem}

By defining
\begin{align*}
\varphi_0&:=\int_{\Gamma}\hf\:{\rm d}\Gamma,
&
\varphi_i&:=\int_{\Gamma}\nu_i \hf\:{\rm d}\Gamma,
&
\varphi_4&:=\int_{\Gamma}(\phi-\alpha)\:{\rm d}\Gamma,
\end{align*}
for $i\in\{1,2,3\}$ and observing that their Frechet derivatives exist and are continuous, linear and bijective it follows from  the Euler-Lagrange Theorem (Theorem 7.15-1 in \cite{ciarlet2013linear}) that if $(\phi,\hf)$ is a solution of Problem \ref{Prob-Euler} then there exists $\lambda\in\mathbb{R}^5$ such that $(\phi,\hf,\lambda)$ is a solution of the problem given below.
\begin{Problem}
Find $(\phi,\hf,\lambda)\in \mathcal{K}\times \mathbb{R}^5$ such that for all $w\in H^2(\Gamma)$ and for all $v\in H^2(\Gamma)$, 
\begin{align}
\label{EL4}
\int_{\Gamma}\left(\frac{b}{\epsilon}W^\prime(\phi)w+b\epsilon\nabla_\Gamma \phi\cdot\nabla_\Gamma w+\frac{2\kappa\Lambda \hf w}{R^2}+\kappa\Lambda\Delta_\Gamma \hf w+\kappa\Lambda^2\phi w+\lambda_0w\right)\:{\rm d}\Gamma&=0,
\\
\label{EL3}	
\begin{split}
\int_{\Gamma}\left(\kappa\Delta_\Gamma \hf\Delta_\Gamma v+\left(\sigma-\frac{2\kappa}{R^2}\right)\nabla_\Gamma \hf\cdot\nabla_\Gamma v-\frac{2\sigma}{R^2}\hf v+\right.\qquad\qquad\qquad\qquad\qquad\:
\\
\left.\kappa\Lambda\phi\Delta_\Gamma v+\frac{2\kappa\Lambda\phi v}{R^2}+\sum_{i=1}^{3}\lambda_iv\nu_i+\lambda_4v\right)\:{\rm d}\Gamma&=0.
\end{split}
\end{align}
\end{Problem} 

By testing with appropriate functions we can determine the values of the Lagrange multipliers $\lambda_i$ for $i\in\{0,1,2,3,4\}$. Testing equation (\ref{EL4}) with 1 it follows that the Lagrange multiplier $\lambda_0$ is given by
\begin{displaymath}
\lambda_0=-\kappa\Lambda^2\alpha-\frac{b}{\epsilon}\dashint_{\Gamma}W^\prime(\phi)\:{\rm d}\Gamma.
\end{displaymath}
Testing equation (\ref{EL3}) with $\nu_i$, and using the fact that  $-\Delta_\Gamma\nu_i=\frac{2}{R^2}\nu_i$ and $\int_{\Gamma}\nu_i\nu_j\:{\rm d}\Gamma=\frac{4\pi R^2}{3}\delta_{ij}$ it follows that
\begin{align*}
\lambda_i&=0\qquad\text{ for }i=1,2,3.
\end{align*}
Finally testing equation (\ref{EL3}) it follows that
\begin{displaymath}
\lambda_4=-\frac{2\kappa\Lambda\alpha}{R^2}.
\end{displaymath}

The PDEs corresponding with \eqref{EL4} and \eqref{EL3} are then given by
\begin{align}
\label{eqn-EL1}
\frac{b}{\epsilon}W^\prime(\phi)-b\epsilon\Delta_\Gamma\phi+\kappa\Lambda\Delta_\Gamma \hf+\frac{2\kappa\Lambda \hf}{R^2}+\kappa\Lambda^2\phi+\lambda_0=&0,
\\
\label{eqn-EL2}
\kappa\Delta_\Gamma^2 \hf-\left(\sigma-\frac{2\kappa}{R^2}\right)\Delta_\Gamma \hf-\frac{2\sigma \hf}{R^2}+\Lambda\Delta_\Gamma\phi+\frac{2\Lambda\phi }{R^2}+\lambda_4=&0.
\end{align}
\subsection{Reduced Order Derivation}
The Euler-Lagrange equations given in \eqref{eqn-EL1} and \eqref{eqn-EL2} can be simplified to a system of two second order equations. We rewrite \eqref{eqn-EL2} as follows
\begin{align}
\label{eqn-EL2*}
\left(\Delta_\Gamma+\frac{2}{R^2}\right)\left(\frac{\sigma}{\kappa}-\Delta_\Gamma \right)\hf=\Lambda\left(\Delta_\Gamma+\frac{2}{R^2}\right)(\phi-\alpha),
\end{align}
and note that if
\begin{align}
\left(\Delta_\Gamma +\frac{2}{R^2}\right)z=0,
\end{align}
then $z$ is an eigenfunction of $-\Delta_\Gamma$ with eigenvalue $\frac{2}{R^2}$ and hence $z\in\text{span}\{\nu_1,\nu_2,\nu_3\}$. Therefore it follows from \eqref{eqn-EL2*} that there exists some $\beta\in\text{span}\{\nu_1,\nu_2,\nu_3\}$ such that
\begin{align}
\label{eqn:red-ord-1}
\left(\frac{\sigma}{\kappa}-\Delta_\Gamma\right)u=\Lambda(\phi-\alpha)+\beta.
\end{align}
Now writing $\mathcal{V}=\text{span}\{1,\nu_1,\nu_2,\nu_3\}^\perp$ it follows from a simple calculation that since $u\in\mathcal{V}$ then $\left(\frac{\sigma}{\kappa}-\Delta_\Gamma\right)u\in\mathcal{V}$ also. Denoting the projection onto $\mathcal{V}$ by $\mathbf{P}$ and applying this projection to \eqref{eqn:red-ord-1} results in
\begin{align}
\label{eqn:red-ord-2}
\left(\frac{\sigma}{\kappa}-\Delta_\Gamma\right)u=\Lambda\mathbf{P}\phi.
\end{align}
This motivates introducing an operator $\mathcal{G}:\mathcal{V}\to \mathcal{V}$ where given $\eta\in\mathcal{V}$, $\mathcal{G}(\eta)$ denotes the unique solution $v\in\mathcal{V}$ of the elliptic equation
\begin{align}
\left(\frac{\sigma}{\kappa}-\Delta_\Gamma \right)v=\Lambda\eta.
\end{align} 
From this and \eqref{eqn:red-ord-2} it follows that
\begin{align}
\label{eqn-reducedHeight}
\hf=\mathcal{G}(\mathbf{P}\phi).
\end{align}
Therefore  we can rewrite \eqref{eqn-EL1} as
\begin{align}
\label{RedOrd PDE}
\frac{b}{\epsilon}\left(W^\prime(\phi)-\dashint_\Gamma W^\prime(\phi)\:{\rm d}\Gamma\right)-b\epsilon\Delta_\Gamma\phi+\kappa\Lambda\left(\Delta_\Gamma +\frac{2}{R^2}\right)\mathcal{G}(\mathbf{P}\phi)+\kappa\Lambda^2(\phi-\alpha)=0,
\end{align}
or equivalently
\begin{align}
\label{RedOrd PDE2}
\frac{b}{\epsilon}\left(W^\prime(\phi)-\dashint_\Gamma W^\prime(\phi)\:{\rm d}\Gamma\right)-b\epsilon\Delta_\Gamma\phi+\kappa\Lambda\mathcal{G}\left(\left(\Delta_\Gamma +\frac{2}{R^2}\right)(\phi-\alpha)\right)+\kappa\Lambda^2(\phi-\alpha)=0.
\end{align}
Using \eqref{eqn-reducedHeight} we can define a new energy $\widetilde{\mathcal{E}}$ given by
\begin{align}
\widetilde{\mathcal{E}}(\phi):=\mathcal{E}(\phi,\mathcal{G}(\phi)),
\end{align}
which simplifies to
\begin{align}
\label{eqn-reducedenergy}
\widetilde{\mathcal{E}}(\phi)=\int_{\Gamma}\frac{\kappa\Lambda}{2}\mathbf{P}\phi\left(\Delta_\Gamma+\frac{2}{R^2}\right)\mathcal{G}(\mathbf{P}\phi)+\frac{b\epsilon}{2}|\nabla_{\Gamma}\phi|^2+
\frac{b}{\epsilon}W(\phi)+\frac{\kappa\Lambda^2\phi^2}{2}\:{\rm d}\Gamma.
\end{align}
We note that if $(\phi^*,u^*)$ is a minimiser of $\mathcal{E}$ then $u^*=\mathcal{G}(\mathbf{P}\phi^*)$ since it is also a critical point and must satisfy \eqref{eqn-EL2*}. Let us further suppose that $\widetilde{\phi}^*$ is a minimiser of $\widetilde{\mathcal{E}}$ then it follows that
\begin{align}
\label{eqn:red-ord-3}
\mathcal{E}(\phi^*,u^*)\leq\mathcal{E}(\widetilde{\phi}^*,\mathcal{G}(\mathbf{P}\widetilde{\phi}^*))=\widetilde{\mathcal{E}}(\widetilde{\phi}^*)\leq\widetilde{\mathcal{E}}(\phi^*)=\mathcal{E}(\phi^*,\mathcal{G}(\mathbf{P}\phi^*))=\mathcal{E}(\phi^*,u^*),
\end{align}
and hence all the inequalities in \eqref{eqn:red-ord-3} are equalities so $\phi^*$ is a minimiser of $\widetilde{\mathcal{E}}$ and $(\widetilde{\phi}^*,\mathcal{G}(\mathbf{P}\widetilde{\phi}^*))$ is a minimiser of $\mathcal{E}$. Therefore we find that finding minimisers of $\widetilde{\mathcal{E}}$ is equivalent to finding minimisers of $\mathcal{E}$.
\section{Gradient Flow}
We observe that the first variation of $\mathcal{E}(\cdot,\cdot)$ is given by
\begin{align*}
\mathcal{E}^\prime(\phi,\hf)[w,v]=\int_{\Gamma}&\frac{b}{\epsilon}W^\prime(\phi)w+b\epsilon\nabla_\Gamma\phi\cdot\nabla_\Gamma w+\left(\sigma-\frac{2\kappa}{R^2}\right)\nabla_\Gamma \hf\cdot\nabla_\Gamma v+\kappa\Delta_\Gamma \hf\Delta_\Gamma v
\\
&-\frac{2\sigma \hf v}{R^2}+\kappa\Lambda w\Delta_\Gamma \hf+\kappa\Lambda\phi\Delta_\Gamma v-\frac{2\kappa\Lambda \hf w}{R^2}-\frac{2\kappa\Lambda \phi v}{R^2}+\kappa\Lambda^2\phi w\:{\rm d}\Gamma.
\end{align*}
We consider the equations
\begin{align}
\label{GF1}
\begin{split}
-\alpha_1(\phi_t,w)_{L^2(\Gamma)}&=\int_{\Gamma}\frac{b}{\epsilon}W^\prime(\phi)w+b\epsilon\nabla_\Gamma\phi\cdot\nabla_\Gamma w
\\
&\qquad+\kappa\Lambda w\Delta_\Gamma \hf +\frac{2\kappa\Lambda \hf w}{R^2}+\kappa\Lambda^2\phi w\:{\rm d}\Gamma,
\end{split}
\\
\label{GF2}
\begin{split}
-\alpha_2(\hf_t,v)_{L^2(\Gamma)}&=\int_{\Gamma}\left(\sigma-\frac{2\kappa}{R^2}\right)\nabla_\Gamma \hf\cdot\nabla_\Gamma v+\kappa\Delta_\Gamma \hf\Delta_\Gamma v
\\
&\qquad-\frac{2\sigma \hf v}{R^2}+\kappa\Lambda\phi\Delta_\Gamma v+\frac{2\kappa\Lambda \phi v}{R^2}\:{\rm d}\Gamma,
\end{split}
\end{align}
for all $v\in V$ and for all $w\in W$, which can be seen to give rise to a gradient flow of $\mathcal{E}(\phi,\hf)$ in $W\times V$ since
\begin{align}
\frac{d}{dt}\mathcal{E}(\phi,\hf)=-\alpha_1\|\phi_t||^2_{L^2(\Gamma)}-\alpha_2\|\hf_t\|^2_{L^2(\Gamma)}\leq 0.
\end{align}
By applying the Euler-Lagrange theorem, and introducing Lagrange multipliers $\lambda_i$ for $i\in\{0,1,2,3,4\}$ this implies that for all $w\in H^1(\Gamma)$ and for all $v\in H^2(\Gamma)$,
\begin{align}
\label{EL5}
\begin{split}
-\alpha_1(\phi_t,w)_{L^2(\Gamma)}&=\int_{\Gamma}\frac{b}{\epsilon}W^\prime(\phi)w+b\epsilon\nabla_\Gamma\phi\cdot\nabla_\Gamma w
\\
&\qquad+\kappa\Lambda w\Delta_\Gamma \hf +\frac{2\kappa\Lambda \hf w}{R^2}+\kappa\Lambda^2\phi w+\lambda_0w\:{\rm d}\Gamma,
\end{split}
\\
\label{EL6}
\begin{split}
-\alpha_2(\hf_t,v)_{L^2(\Gamma)}&=\int_{\Gamma}\left(\sigma-\frac{2\kappa}{R^2}\right)\nabla_\Gamma \hf\cdot\nabla_\Gamma v+\kappa\Delta_\Gamma \hf\Delta_\Gamma v
\\
&\qquad-\frac{2\sigma \hf v}{R^2}+\kappa\Lambda\phi\Delta_\Gamma v+\frac{2\kappa\Lambda \phi v}{R^2}+\sum_{i=1}^{3}\lambda_iv\nu_i+\lambda_4v\:{\rm d}\Gamma,
\end{split}
\end{align}
where $\lambda_i$ for $i\in\{0,1,2,3\}$ are Lagrange multipliers. Testing equation \eqref{EL5} with $1$ and equation \eqref{EL6} with $1,\nu_1,\nu_2$ and $\nu_3$ as in subsection \ref{332}, we observe that the Lagrange multipliers $\lambda_i$ for $i\in\{0,1,2,3,4\}$ are again given by
\begin{align}
\lambda_0&=-\kappa\Lambda^2\alpha-\frac{b}{\epsilon}\dashint_\Gamma W^\prime(\phi)\:{\rm d}\Gamma,&&
\lambda_1=\lambda_2=\lambda_3=0,&&
\lambda_4=-\frac{2\kappa\Lambda\alpha}{R^2}.
\end{align}
Hence, a gradient flow of $\mathcal{E}(\cdot,\cdot)$ in $W\times V$ is given by
\begin{align}
\label{AC PDE}
\begin{cases}
\alpha_1\phi_t+\frac{b}{\epsilon}W^\prime(\phi)-b\epsilon\Delta_\Gamma\phi+\kappa\Lambda\Delta_\Gamma \hf+\frac{2\kappa\Lambda \hf}{R^2}+\kappa\Lambda^2\phi+\lambda_0=0&\Gamma\times(0,T),\\
\alpha_2 \hf_t-\left(\sigma-\frac{2\kappa}{R^2}\right)\Delta_\Gamma \hf+\kappa\Delta_\Gamma^2 \hf-\frac{2\sigma \hf}{R^2}+\kappa\Lambda\Delta_\Gamma\phi+\frac{2\kappa\Lambda \phi}{R^2}+\lambda_4=0&\Gamma\times (0,T),\\
\phi(\cdot,0)=\phi_0(\cdot)&\Gamma\times\{t=0\},\\
\hf(\cdot,0)=\hf_0(\cdot)&\Gamma\times\{t=0\}.
\end{cases}
\end{align}
%
\subsection{Existence}
Before turning to consider numerical simulations of \eqref{AC PDE}, we first address questions related to well-posedness. Beginning with existence we will prove the following result.
\begin{theorem}
\label{Existence}
Suppose $(\phi_0,\hf_0)\in\mathcal{K}$, then there exists $(\phi,\hf)\in\mathcal{K}$ such that
\begin{align*}
\phi&\in L^\infty(0,T;H^1(\Gamma))\cap C([0,T];L^2(\Gamma)),
\\
\hf&\in L^\infty(0,T; H^2(\Gamma))\cap C([0,T];L^2(\Gamma)),
\\
\phi^\prime&\in L^2(0,T;L^2(\Gamma)),
\\
\hf^\prime&\in L^2(0,T;L^2(\Gamma)),
\\
\hf_0&=\hf(0),
\\
\phi_0&=\phi(0),
\end{align*}
and satisfying
\begin{align}
\begin{split}
-\int_{0}^{T}\alpha_1\left<\phi^\prime,\eta\right>\:{\rm dt}=\int_0^T\left[\int_{\Gamma}\right.&\frac{b}{\epsilon}\left(W^\prime(\phi)-\dashint_\Gamma W^\prime(\phi)\:{\rm d}\Gamma\right)\eta+b\epsilon\nabla_\Gamma\phi\cdot\nabla_\Gamma\eta
\\
&\left.-\kappa\Lambda\nabla_\Gamma \hf\cdot\nabla_\Gamma\eta+\frac{2\kappa\Lambda \hf\eta}{R^2}+\kappa\Lambda^2(\phi-\alpha)\eta\:{\rm d}\Gamma	\right]\:{\rm dt},
\end{split}
\\
\begin{split}
-\int_{0}^{T}\alpha_2\left<\hf^\prime,\xi\right>\:{\rm dt}=\int_0^T\left[\int_{\Gamma}\right.& \kappa\Delta_\Gamma \hf\Delta_\Gamma\xi+\left(\sigma-\frac{2\kappa}{R^2}\right)\nabla_\Gamma \hf\cdot\nabla_\Gamma\xi
\\
&\left.-\frac{2\sigma \hf\xi}{R^2}+\frac{2\kappa\Lambda(\phi-\alpha)\xi}{R^2}-\kappa\Lambda\nabla_\Gamma\phi\cdot\nabla_\Gamma\xi\:{\rm d}\Gamma\right]\:{\rm dt},
\end{split}
\end{align}
for all $\eta\in L^2(0,T;H^1(\Gamma))$ and for all $\xi\in L^2(0,T;H^2(\Gamma))$.
\end{theorem}
\subsubsection{Galerkin problem}

We prove Theorem \ref{Existence} using a Galerkin method. Using that there exist smooth eigenfunctions $\left\{z_j\right\}$ of the Laplace-Beltrami operator $-\Delta_\Gamma$ which form an orthonormal basis of $H^1(\Gamma)$ and are orthogonal in $L^2(\Gamma)$, we define $V^m$ as
\begin{displaymath}
V^m:=span\left\{z_1,z_2,...,z_m\right\},
\end{displaymath}
and set $\mathcal{P}_m:L^2(\Gamma)\to V^m$ to be the Galerkin projection given by
\begin{displaymath}
(P_mv-v,u_m)=0\qquad\forall v\in L^2(\Gamma), u_m\in V^m.
\end{displaymath}
$P_m$ then satisifies the following strong convergence results,
\begin{align}
\label{eqn-pm1}
\mathcal{P}_mv\to& v\text{ in }L^2(\Gamma)\quad\forall v\in L^2(\Gamma),\\
\mathcal{P}_mv\to& v\text{ in }H^1(\Gamma)\quad\forall v\in H^1(\Gamma),\\
\mathcal{P}_mv\to& v\text{ in }H^2(\Gamma)\quad\forall v\in H^2(\Gamma).
\label{eqn-pm3}
\end{align}

Therefore, the Galerkin system we are considering is given by
\begin{align}
\label{discretephi}
\begin{split}
-\alpha_1\left<\phi_m^\prime,\eta_m\right>=\int_{\Gamma}&\frac{b}{\epsilon}\left(W^\prime(\phi_m)-\dashint_\Gamma W^\prime(\phi_m)\right)\eta_m+b\epsilon\nabla_\Gamma\phi_m\cdot\nabla_\Gamma\eta_m
\\
&-\kappa\Lambda\nabla_\Gamma \hf_m\cdot\nabla_\Gamma\eta_m+\frac{2\kappa\Lambda \hf_m\eta_m}{R^2}+\kappa\Lambda^2(\phi_m-\alpha)\eta_m\:{\rm d}\Gamma,
\end{split}
\\
\label{discreteh}
\begin{split}
-\alpha_2\left<\hf_m^\prime,\xi_m\right>=\int_{\Gamma}& \kappa\Delta_\Gamma \hf_m\Delta_\Gamma\xi_m+\left(\sigma-\frac{2\kappa}{R^2}\right)\nabla_\Gamma \hf_m\cdot\nabla_\Gamma\xi_m
\\
&-\frac{2\sigma \hf_m\xi_m}{R^2}+\frac{2\kappa\Lambda(\phi_m-\alpha)\xi_m}{R^2}-\kappa\Lambda\nabla_\Gamma\phi_m\cdot\nabla_\Gamma\xi_m\:{\rm d}\Gamma,
\end{split}
\end{align}
for all $\eta_m, \xi_m \in V^m$.

This system can then be written as an initial value problem for a system of ordinary differential equations with locally Lipschitz right hand sides, for which there exists a unique solution at least locally in time.

We observe that 
\begin{displaymath}
\left<\phi_m^\prime,\eta_m\right>=(\phi_m^\prime,\eta_m)\qquad\text{and}\qquad\left<\hf_m^\prime,\mu_m\right>=(\hf_m^\prime,\mu_m).
\end{displaymath}
Testing \eqref{discretephi} and \eqref{discreteh} with $\eta_m=1$ and $\xi_m=1,\nu_1,\nu_2,\nu_3$, and applying standard ODE results  it follows that if $(\phi_m(0),\hf_m(0))\in\mathcal{K}$ then the solution $(\phi_m(t),\hf_m(t))\in\mathcal{K}$ for $t\in[0,T]$, where $T$ comes from the local existence result used above.
\subsubsection{Energy estimates}
In order to pass to the limit, and prove  existence of the  full system we derive some \emph{a priori} estimates by considering the discrete energy $\mathcal{E}(\phi_m,\hf_m)$.
\begin{theorem}
\label{Energy Estimates}
Suppose $(\phi_m, \hf_m)\in\mathcal{K}$ satisfy equations ~\eqref{discretephi}~--~\eqref{discreteh} then there exists a constant $C$ independent of $m$ such that
\begin{align}
\label{En Est 1}
\|\phi_m\|_{L^\infty(0,T;H^1(\Gamma))}\leq C,
\\
\label{En Est 3}
\|\hf_m\|_{L^\infty(0,T;H^2(\Gamma))}\leq C,
\\
\label{En Est 4}
\|\phi^\prime_m\|_{L^2(0,T;L^2(\Gamma))}\leq C,
\\
\label{En Est 5}
\|\hf^\prime_m\|_{L^2(0,T;L^2(\Gamma))}\leq C,
\end{align} 
\end{theorem}
\begin{proof}
By differentiating the energy functional $\mathcal{E}(\cdot,\cdot)$ with respect to $t$ we obtain,
\begin{align}
\frac{d}{dt}\mathcal{E}(\phi_m,\hf_m)=-\alpha_1\|\phi_m^\prime\|^2_{L^2(\Gamma)}-\alpha_2\|\hf_m^\prime\|^2_{L^2(\Gamma)}.
\end{align}
Integrating and using the coercivity of $\mathcal{E}(\cdot,\cdot)$ proven in Proposition \ref{Prop-Direct} it follows that for all $t\in(0,T)$,
\begin{align}
\|\hf_m\|^2_{H^2(\Gamma)}+\|\phi_m\|^2_{H^1(\Gamma)}+\int_{0}^{t}\|\phi^\prime_m\|^2_{L^2(\Gamma)}\:{\rm dt}+\int_{0}^{t}\|\hf_m^\prime\|^2_{L^2(\Gamma)}\:{\rm dt}\leq C,
\end{align}
where in the above line we have used that $\mathcal{E}(\phi_m(0),\hf_m(0))\leq C$ where $C$ is some constant independent of $m$. From which it follows that for all $t\in(0,T)$,
\begin{align}
\sup_{t\in(0,T)}\|\hf_m\|^2_{H^2(\Gamma)}+\sup_{t\in(0,T)}\|\phi_m\|^2_{H^1(\Gamma)}+\int_{0}^{t}\|\phi^\prime_m\|^2_{L^2(\Gamma)}\:{\rm dt}+\int_{0}^{t}\|\hf_m^\prime\|^2_{L^2(\Gamma)}\:{\rm dt}\leq C
\end{align}
which give the required energy bounds.
\end{proof}
\subsubsection{Existence theorem proof}
Applying the energy estimates proven in Theorem \ref{Energy Estimates}
and considering subsequences as neccessary, there exist $\phi^*$ and $\hf^*$ in the indicated spaces such that the following convergence results hold in the weak sense,
\begin{align}
\label{eqn-weakprime}
\phi_m^\prime\rightharpoonup\left(\phi^*\right)^\prime&\text{ in }L^2(0,T;L^2(\Gamma)),&\hf_m^\prime\rightharpoonup\left(\hf^*\right)^\prime&\text{ in }L^2(0,T;L^2(\Gamma)),
\\
\phi_m\rightharpoonup \phi^*&\text{ in }L^2(0,T;H^1(\Gamma)),&\hf_m\rightharpoonup \hf^*&\text{ in }L^2(0,T;H^2(\Gamma)),
\end{align}
and applying standard compactness results (Aubin-Lions Lemma and Kondrachov's Theorem) the following convergence results hold in the strong sense,
\begin{align}
\phi_m\to\phi^*&\text{ in }C([0,T];L^2(\Gamma)),&\hf_m\to \hf^*&\text{ in }C([0,T];L^2(\Gamma)),
\\
\phi_m\to \phi^*&\text{ in }L^2(0,T;L^p(\Gamma)),
\label{eqn-strongkon}
\end{align}
where $p\geq 1$. Furthermore since $\phi_m(0)\to\phi^*(0)$ and $\hf_m(0)\to \hf^*$ in $L^2(\Gamma)$ it holds that
\begin{align}
\phi^*(0)&=\phi_0,&\hf^*(0)&=\hf_0.
\end{align}

Taking $\eta\in L^2(0,T;H^1(\Gamma))$,  and $\xi\in L^2(0,T;H^2(\Gamma))$ we have that
\begin{align}
\begin{split}
-\int_{0}^{T}&\alpha_1\left<\phi_m^\prime,\mathcal{P}_m\eta\right>\:{\rm dt}\\
=\int_0^T&\left[\int_{\Gamma}\right.\frac{b}{\epsilon}\left(W^\prime(\phi_m)-\dashint_\Gamma W^\prime(\phi_m)\:{\rm d}\Gamma\right)\mathcal{P}_m\eta+b\epsilon\nabla_\Gamma\phi_m\cdot\nabla_\Gamma\mathcal{P}_m\eta
\\
&
\left.\quad-\kappa\Lambda\nabla_\Gamma \hf_m\cdot\nabla_\Gamma\mathcal{P}_m\eta+\frac{2\kappa\Lambda \hf_m\mathcal{P}_m\eta}{R^2}+\kappa\Lambda^2(\phi_m-\alpha)\mathcal{P}_m\eta\:{\rm d}\Gamma	\right]\:{\rm dt},
\end{split}
\end{align}
and
\begin{align}
\begin{split}
-\int_{0}^{T}&\alpha_2\left<\hf_m^\prime,\mathcal{P}_m\xi\right>\:{\rm dt}\\
=\int_0^T&\left[\int_{\Gamma}\right.\kappa\Delta_\Gamma \hf_m\Delta_\Gamma\mathcal{P}_m\xi+\left(\sigma-\frac{2\kappa}{R^2}\right)\nabla_\Gamma \hf_m\cdot\nabla_\Gamma\mathcal{P}_m\xi
\\
&\left.\quad-\frac{2\sigma \hf_m\mathcal{P}_m\xi}{R^2}+\frac{2\kappa\Lambda(\phi_m-\alpha)\mathcal{P}_m\xi}{R^2}
-\kappa\Lambda\nabla_\Gamma\phi_m\cdot\nabla_\Gamma\mathcal{P}_m\xi\:{\rm d}\Gamma\right]\:{\rm dt},
\end{split}
\end{align}

Using the convergence results \eqref{eqn-pm1}-\eqref{eqn-pm3} and \eqref{eqn-weakprime}-\eqref{eqn-strongkon} we can pass to the limit to obtain
\begin{align}
\begin{split}
-\int_{0}^{T}\alpha_1\left<(\phi^*)^\prime,\eta\right>\:{\rm dt}=\int_0^T\left[\int_{\Gamma}\right.&\frac{b}{\epsilon}\left(W^\prime(\phi^*)-\dashint_\Gamma W^\prime(\phi^*)\:{\rm d}\Gamma\right)\eta+b\epsilon\nabla_\Gamma\phi^*\cdot\nabla_\Gamma\eta
\\
&\left.-\kappa\Lambda\nabla_\Gamma \hf^*\cdot\nabla_\Gamma\eta
+\left.\frac{2\kappa\Lambda \hf^*\eta}{R^2}\right.
+\kappa\Lambda^2(\phi^*-\alpha)\eta\:{\rm d}\Gamma\right]\:{\rm dt},
\end{split}
\end{align}
\begin{align}
\begin{split}
-\int_{0}^{T}\alpha_2\left<(\hf^*)^\prime,\xi\right>\:{\rm dt}=\int_0^T\left[\int_{\Gamma}\right. &\kappa\Delta_\Gamma \hf^*\Delta_\Gamma\xi+\left(\sigma-\frac{2\kappa}{R^2}\right)\nabla_\Gamma \hf^*\cdot\nabla_\Gamma\xi
\\
&\left.-\frac{2\sigma \hf^*\xi}{R^2}+\frac{2\kappa\Lambda(\phi^*-\alpha)\xi}{R^2}-\kappa\Lambda\nabla_\Gamma\phi^*\cdot\nabla_\Gamma\xi\:{\rm d}\Gamma\right]\:{\rm dt},
\end{split}
\end{align}
$\forall\eta\in L^2(0,T;H^1(\Gamma))$, and $\forall\xi\in L^2(0,T;H^2(\Gamma))$. This completes the proof of Theorem \ref{Existence}.
\subsection{Uniqueness Theory}
\begin{theorem}[Uniqueness]
	\label{Uniqueness}
There exists at most one solution pair.
\end{theorem}
\begin{proof}
Let $(\phi_i,\hf_i)$, $i=1,2$ be two solution pairs. Set $\theta^\phi=\phi_1-\phi_2$ and $\theta^\hf=\hf_1-\hf_2$. By subtracting the equations, testing with $\eta=\theta^\phi$ and $\xi=\theta^\hf$ and using that
\begin{align}
\frac{d}{dt}\|\theta^\phi\|^2_{L^2(\Gamma)}&=2\left<\left(\theta^\phi\right)^\prime,\theta^\phi\right>,&\frac{d}{dt}\|\theta^\hf\|^2_{L^2(\Gamma)}&=2\left<\left(\theta^\hf\right)^\prime,\theta^\hf\right>,
\end{align}
for a.e. $0\leq t\leq T$
 we obtain
 \begin{align}
 \begin{split}
  -\frac{\alpha_1}{2}\frac{d}{dt}\|\theta^\phi\|^2_{L^2(\Gamma)}=&\int_{\Gamma}\frac{b}{\epsilon}\left(W^\prime(\phi^1)-W^\prime(\phi^2)\right)\theta^\phi\:{\rm d}\Gamma
  +b\epsilon\|\nabla_\Gamma \theta^\phi\|^2_{L^2(\Gamma)}
  \\
  &+\kappa\Lambda^2\|\theta^\phi\|^2_{L^2(\Gamma)}+\int_{\Gamma}\frac{2\Lambda\kappa\theta^\hf\theta^\phi}{R^2}-\Lambda\kappa\nabla_\Gamma\theta^\phi\cdot\nabla_\Gamma\theta^\hf\:{\rm d}\Gamma,
 \end{split}
 \end{align}
 \begin{align}
 \begin{split}
 -\frac{\alpha_2}{2}\frac{d}{dt}\|\theta^\hf\|^2_{L^2(\Gamma)}=&\kappa\|\Delta_\Gamma \theta^\hf\|^2_{L^2(\Gamma)}+\left(\sigma-\frac{2\kappa}{R^2}\right)\|\nabla_\Gamma\theta^\hf\|^2_{L^2(\Gamma)}
 \\
 &-\frac{2\sigma}{R^2}\|\theta^\hf\|^2_{L^2(\Gamma)}+\int_{\Gamma}\frac{2\Lambda\kappa\theta^\hf\theta^\phi}{R^2}-\Lambda\kappa\nabla_\Gamma\theta^\phi\cdot\nabla_\Gamma\theta^\hf\:{\rm d}\Gamma.
 \end{split}
 \end{align}

Using the Poincare type inequality \eqref{eqn-poincare}, structural property (2) of $W(\cdot)$ and Youngs inequality we obtain,
\begin{equation}
\frac{d}{dt}\left(\|\theta^\hf\|^2_{L^2(\Gamma)}+\|\theta^\phi\|^2_{L^2(\Gamma)}\right)+c_1\|\theta^\hf\|^2_{H^2(\Gamma)}+c_2\|\theta^\phi\|^2_{L^2(\Gamma)}
\leq C\left(\|\theta^\hf\|^2_{L^2(\Gamma)}+\|\theta^\phi\|^2_{L^2(\Gamma)}\right),
\end{equation}
where $c_1,c_2$ and $C$ are strictly positive constants. Uniqueness then follows by Gronwall's inequality.
\end{proof}
\subsection{Gradient Flow for the reduced energy}
Returning to consider the reduced energy \eqref{eqn-reducedenergy}, we can likewise obtain the gradient flow equation
\begin{align}
\label{RedGradFlow}
\alpha_1\phi_t+\frac{b}{\epsilon}\left(W^\prime(\phi)-\dashint_\Gamma W^\prime(\phi)\:{\rm d}\Gamma\right)-b\epsilon\Delta_\Gamma\phi+\kappa\Lambda\left(\Delta_\Gamma +\frac{2}{R^2}\right)\mathcal{G}(\mathbf{P}\phi)+\kappa\Lambda^2(\phi-\alpha)=0,
\end{align}
satisfying
\begin{align}
\frac{d}{dt}\widetilde{\mathcal{E}}(\phi)=-\alpha_1\|\phi_t||^2_{L^2(\Gamma)}\leq 0.
\end{align}
However, by defining $u=\mathcal{G}(\mathbf{P}\phi)$ as in \eqref{eqn-reducedHeight} then we obtain the system of equations
\begin{align}
\label{Reduced PDE}
\begin{split}
\alpha_1\phi_t+\frac{b}{\epsilon}\left(W^\prime(\phi)-\dashint_\Gamma W^\prime(\phi)\:{\rm d}\Gamma\right)-b\epsilon\Delta_\Gamma\phi+\kappa\Lambda\left(\Delta_\Gamma +\frac{2}{R^2}\right)u+\kappa\Lambda^2(\phi-\alpha)=0,\\
-\Delta_\Gamma u+\frac{\sigma}{\kappa}u=\Lambda\mathbf{P}\phi
\end{split}
\end{align}
which coincides with \eqref{AC PDE} in the case $\alpha_2=0$. In this instance we can again apply a Galerkin approximation and obtain the \emph{a priori} bounds
\begin{align}
\|\phi_m\|_{L^\infty(0,T;H^1(\Gamma))}\leq C,
\\
\|\hf_m\|_{L^\infty(0,T;H^2(\Gamma))}\leq C,
\\
\|\phi^\prime_m\|_{L^2(0,T;L^2(\Gamma))}\leq C.
\end{align}
From these estimates existence and uniqueness can be shown analagously to Theorem \ref{Existence} and Theorem \ref{Uniqueness}. 
The case $\alpha_2=0$ can be physically understood as instantaneous relaxation of the surface energy.

\section{Numerical Simulations}
In this section we present some numerical results for the longtime behaviour  of the system of PDEs given by \eqref{Reduced PDE}. We suppose the double well potential is given by
\begin{align}
W(r)=\frac{1}{4}(r^2-1)^2.
\end{align}
This choice of $W(\cdot)$ satisifies the structural assumptions given earlier. 

\subsection{Numerical Scheme}
We implement an iterative method as follows. Given a solution $\left(\phi^{(n)},\hf^{(n)}\right)$ at the previous time 
step we consider a sequence $\{\phi_k,\hf_k,\lambda_k\}_{k=1}^\infty$ where $(\phi_k,\hf_k)$ is a solution to
\begin{align}
\label{secant1}
\begin{split}
\int_{\Gamma}&\frac{\phi_k-\phi^{(n)}}{\tau}\eta+\frac{b}{\epsilon}W^{\prime\prime}\left(\phi^{(n)}\right)\left(\phi_k-\phi^{(n)}\right)
+\frac{b}{\epsilon}W^\prime\left(\phi^{(n)}\right)\eta
\\
&+b\epsilon\nabla_\Gamma \phi_k\cdot\nabla_\Gamma\eta-\kappa\Lambda\nabla_\Gamma \hf_k\cdot\nabla_\Gamma\eta+\frac{2\kappa\Lambda}{R^2}\hf_k\eta-\lambda_k\eta+\Lambda^2\kappa(\phi_k-\alpha)\:{\rm d}\Gamma=0,
\end{split}
\end{align}
\begin{align}
\label{secant2}
\begin{split}
\int_{\Gamma}\frac{\sigma}{\kappa}u_k\chi+\nabla_\Gamma u_k\cdot\nabla_\Gamma\chi-\Lambda(\phi_k-\alpha)\chi\:{\rm d}\Gamma&=0,
\end{split}
\end{align}
where in the above, a linearisation has been used for $W^\prime$. The mean value constraint on the height function 
is directly enforced by \eqref{secant2} provided $\sigma\neq 0$. The mean value constraint on $\phi$ is imposed by the secant method, (following \cite{BloEll93-a}),  using the sequence $\{\lambda_{k}\}_{k\geq 1}$ which is constructed  as follows
\begin{displaymath}
\lambda_{k+1}=\lambda_k+\frac{(\lambda_k-\lambda_{k-1})\left(\alpha-\int_{\Gamma}\phi_k\right)}{\left(\int_{\Gamma}\phi_k-\int_{\Gamma}\phi_{k-1}\right)}.
\end{displaymath}
with $\lambda_1=-\frac{b}{\epsilon}$ and $\lambda_2=\frac{b}{\epsilon}$. We stop the iteration when $|\lambda_{k+1}-\lambda_k|<tol$ and set $\phi^{(n+1)}=\phi_{k+1}$ and $\hf^{(n+1)}=\mathbf{P}\hf_{k+1}$. We note that it is not neccessary to consider $\mathbf{P}u_k$ in order to obtain $\phi_k$ since $\left(\Delta_\Gamma+\frac{2}{R^2}\right)\mathbf{P}u_k=\left(\Delta_\Gamma+\frac{2}{R^2}\right)u_k$.

DUNE software was used to implemement a surface finite element method. Specifically we used a PYTHON module (c.f. \cite{dedner2018dune}) which implemented a GMRES method with ILU preconditing to solve the system of linear equations \eqref{secant1}-\eqref{secant2}. For the secant iteration we set $tol=10^{-8}$ and for the GMRES iteration we set the residual tolerance and absolute tolerance both to $10^{-10}$. For the case $\sigma=0$ we additionally used a nullspace method from PETSc \cite{dalcin2011parallel,petsc-user-ref,petsc-efficient}.

Unless stated otherwise, we used a base grid containing 1026 vertices, and at each time step applied an adaptive grid method on each element $K$ if the condition
\begin{align}
\label{eqn:adaptive-grid}
\|\nabla\phi\|_{L^\infty(K)}>\frac{\mu\epsilon}{|K|},
\end{align} 
is satisfied, where $\mu=0.05$. For most of our simulations we will use $\epsilon=0.02$ which typically leads to a grid consisting of around 30,000 vertices. Figure  \ref{fig:adaptive-grid} illustrates an example of such a grid around an interface.
\begin{figure}
	\centering
	\includegraphics[width=0.4\linewidth]{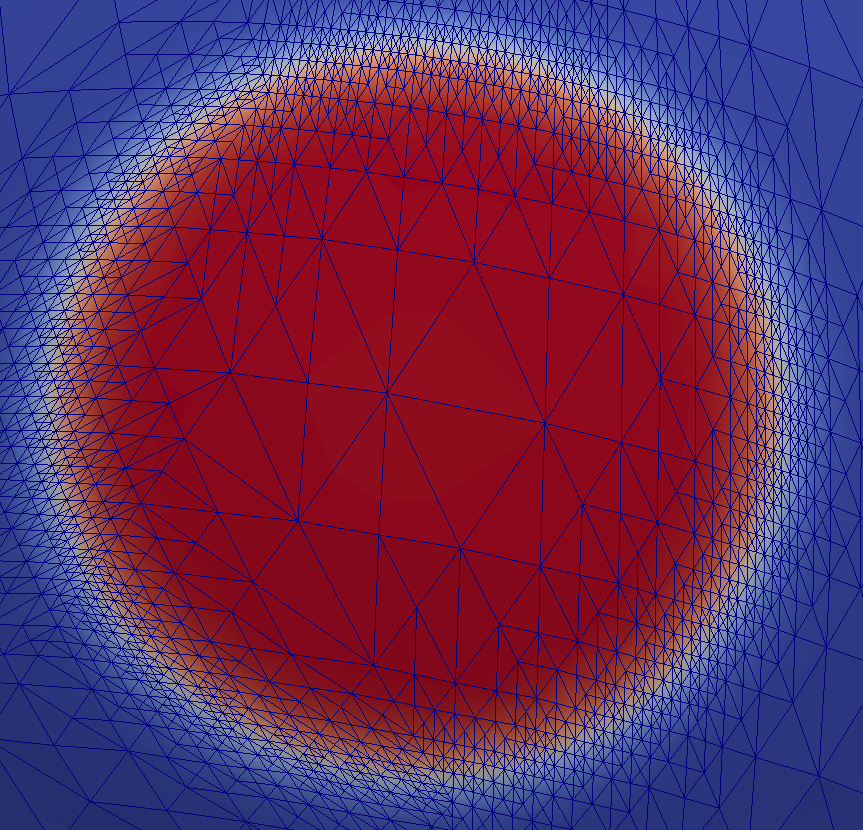}
	\caption{An example of how the adaptive grid method given in \eqref{eqn:adaptive-grid} resolves the interface for the case $\epsilon=0.02$.}
	\label{fig:adaptive-grid}
\end{figure}

 We also used an adaptive time stepping strategy initially using a uniform time step while phase separation occured and then using an adaptive time step (within bounds) that is inversely proportional to
\begin{align}
\max_{x\in\Gamma_h}\frac{\left|\phi_h^{(m)}(x)-\phi_h^{(m-1)}(x)\right|}{\tau^{(m)}},
\end{align}
which should be interpreted as the normal velocity of the interface. 

To graphically represent the numerical solutions, we deform the surface as described by \eqref{eqn-deform}. Here, for visualisation purposes we exagerate the size of the deformation $u_h$ by setting $\rho=1$ whereas in reality it should be significantly smaller. The colouring of the resulting surface is given by $\phi_h$ with red indicating $+1$ regions and blue $-1$ regions.
\subsection{Stabilisation of multiple domains}
We first explore whether there exists stable steady state solutions composed of multiple lipid rafts ($+1$ phase domains), a property observed in biological membranes. We choose $\kappa=1, R=1, b=1, \epsilon=0.02$ and $\sigma=10$, and use a uniform time step of $\tau= 10^{-2}$. We choose initial conditions with an increasing number of lipid rafts and investigate the impact of varying the spontaneous curvature $\Lambda$, which acts as the coupling parameter between the phasefield and the deformation. In each case the inital conditions are chosen such that $\alpha=-0.5$.

\begin{figure}
	\centering
	\begin{subfigure}{.24\linewidth}
		\centering
		\includegraphics[width=\linewidth]{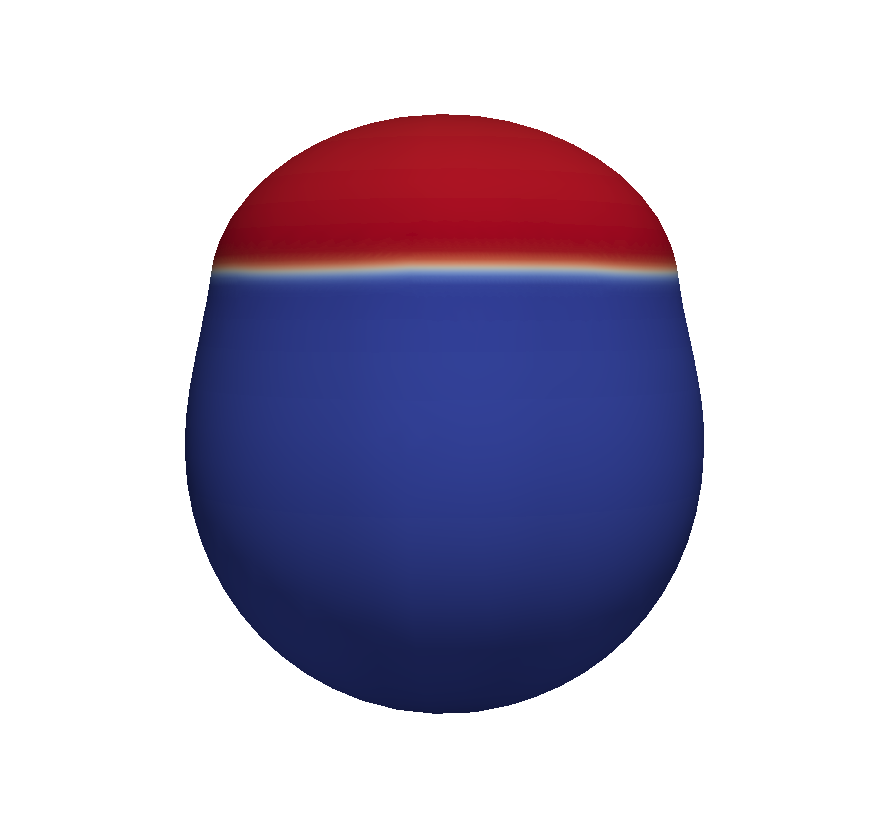}
		\caption{N=1}
	\end{subfigure}
	\begin{subfigure}{.24\linewidth}
		\centering
		\includegraphics[width=\linewidth]{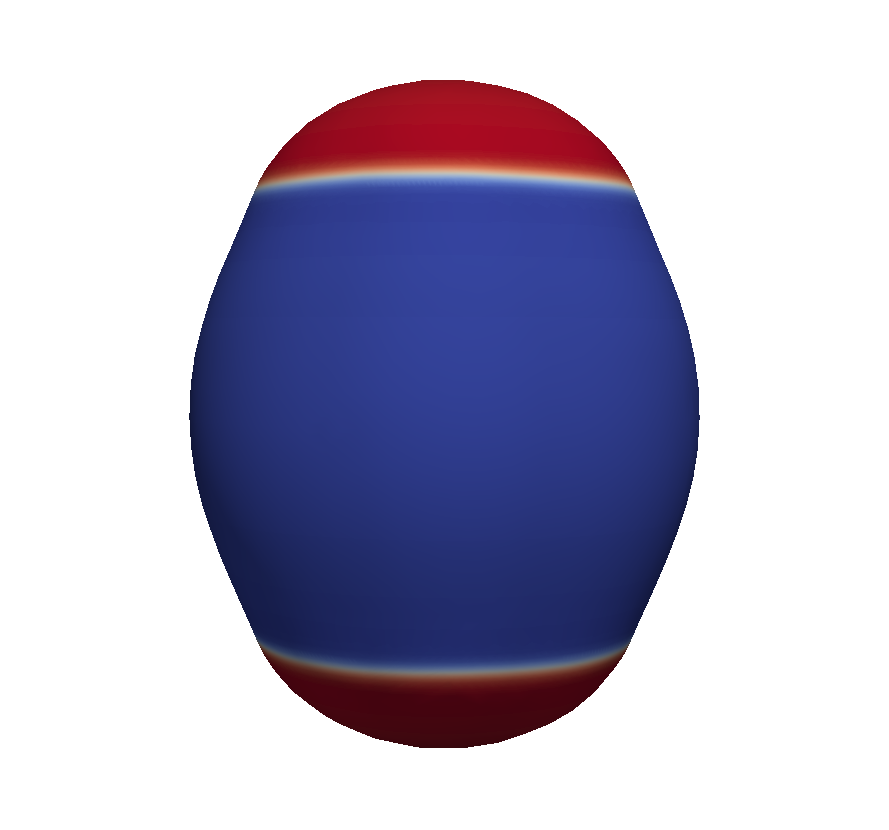}
		\caption{N=2}
	\end{subfigure}
	\begin{subfigure}{.24\linewidth}
		\centering
		\includegraphics[width=\linewidth]{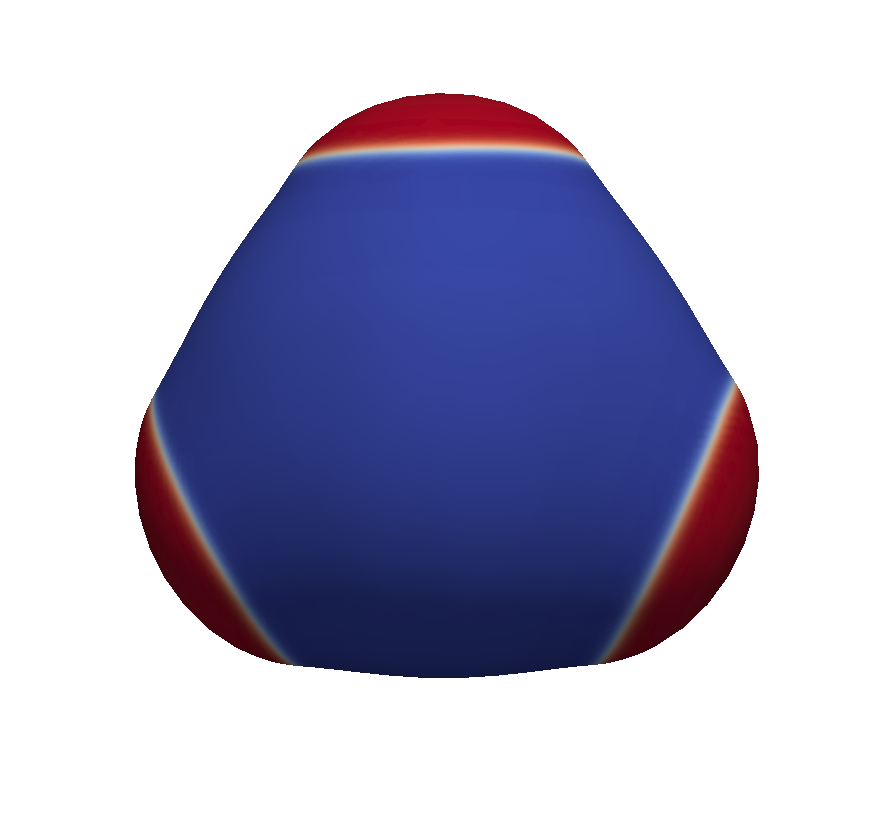}
		\caption{N=3}
	\end{subfigure}
	\begin{subfigure}{.24\linewidth}
		\centering
		\includegraphics[width=\linewidth]{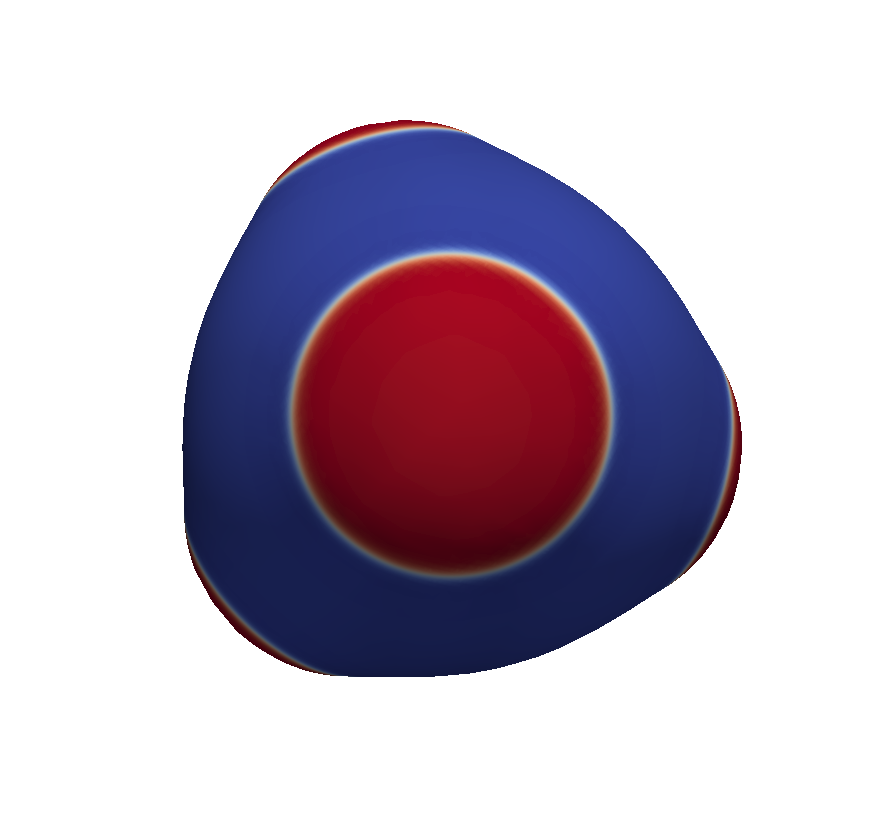}
		\caption{N=4}
	\end{subfigure}
	\caption{Stablised steady states solutions of N domains for $\Lambda=2$.}
	\label{fig-Stabilisation}
\end{figure}
\begin{figure}
	\centering
	\includegraphics[width=0.45\linewidth]{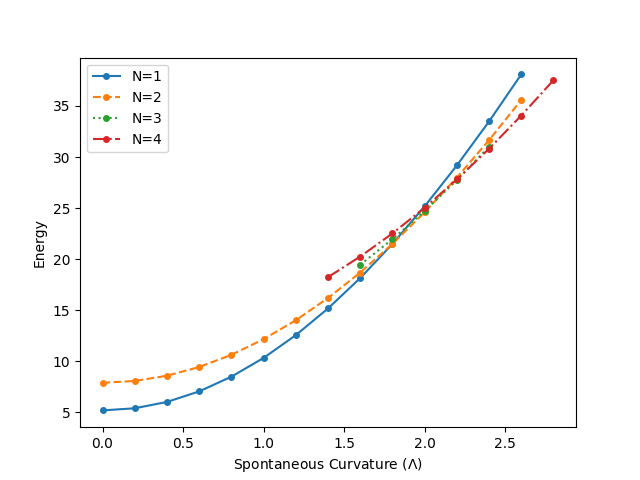}
	\includegraphics[width=0.45\linewidth]{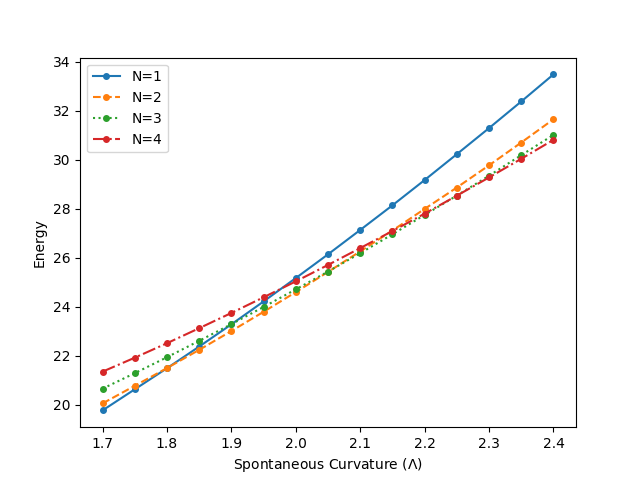}
	\caption{Energy dependence of steady state solutions of $N$ lipid raft domains on spontaneous curvature, $\Lambda$. The graph on the right is a zoomed in version of the graph on the left on an area of interest.}
	\label{fig-Energy-Lambda}
\end{figure}

In Figure \ref{fig-Stabilisation} we depict stabilised steady state solutions consisting of $N$ lipid raft domains for $\Lambda=2$. In  Figure \ref{fig-Energy-Lambda} we plot the energy \eqref{eqn-peturb-energy} against spontaneous curvature $\Lambda$ for the corresponding steady state solutions. The $\Lambda$ values considered were $0, 0.2, 0.4, 0.6 ...$. This was not possible in all cases. For each $\Lambda$ value where no corresponding energy $\mathcal{E}$ has been plotted in Figure \ref{fig-Energy-Lambda} indicates that a state consiting of $N$ lipid raft domains was not a steady state solution. For example the case $N=1$ and $\Lambda=2.8$ is illustrated in Figure \ref{fig-instability}.
\begin{figure}
	\centering
	\begin{subfigure}{.24\linewidth}
		\centering
		\includegraphics[width=\linewidth]{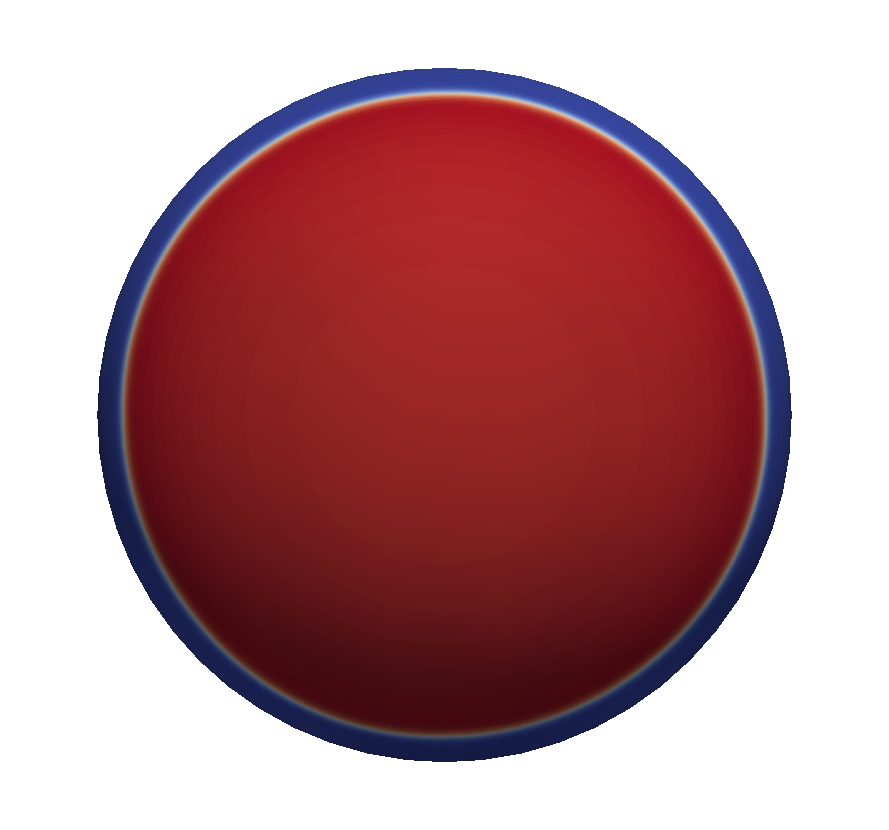}
		\caption{$\phi_h(\cdot,t=0)$}
	\end{subfigure}
	\begin{subfigure}{.24\linewidth}
		\centering
		\includegraphics[width=\linewidth]{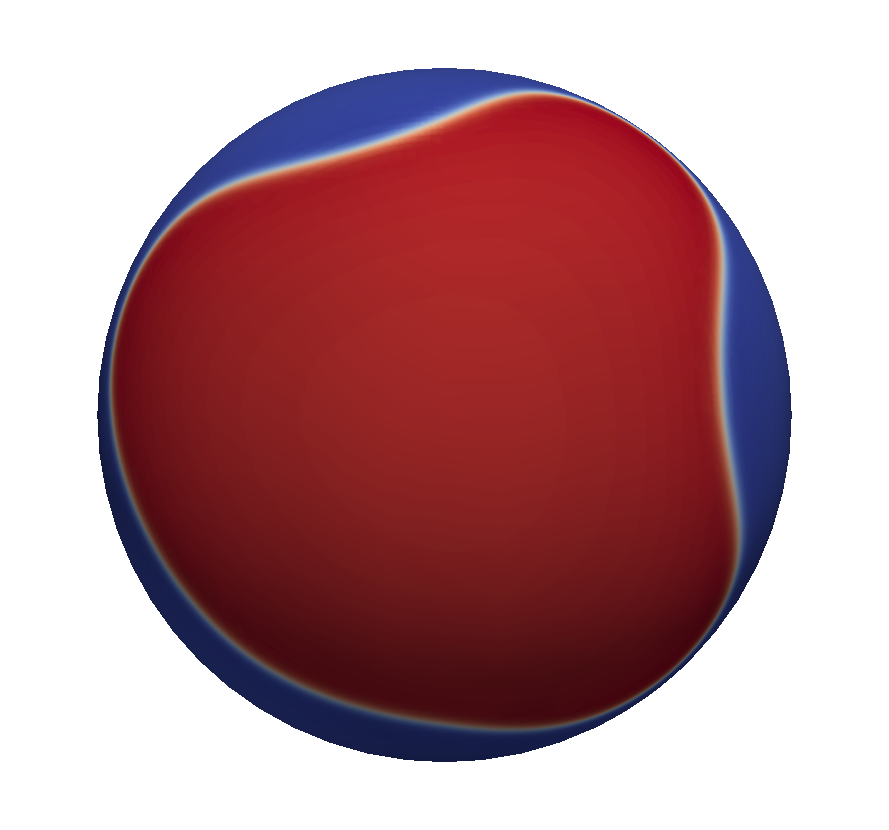}
		\caption{$\phi_h(\cdot,t=60)$}
	\end{subfigure}
	\begin{subfigure}{.24\linewidth}
		\centering
		\includegraphics[width=\linewidth]{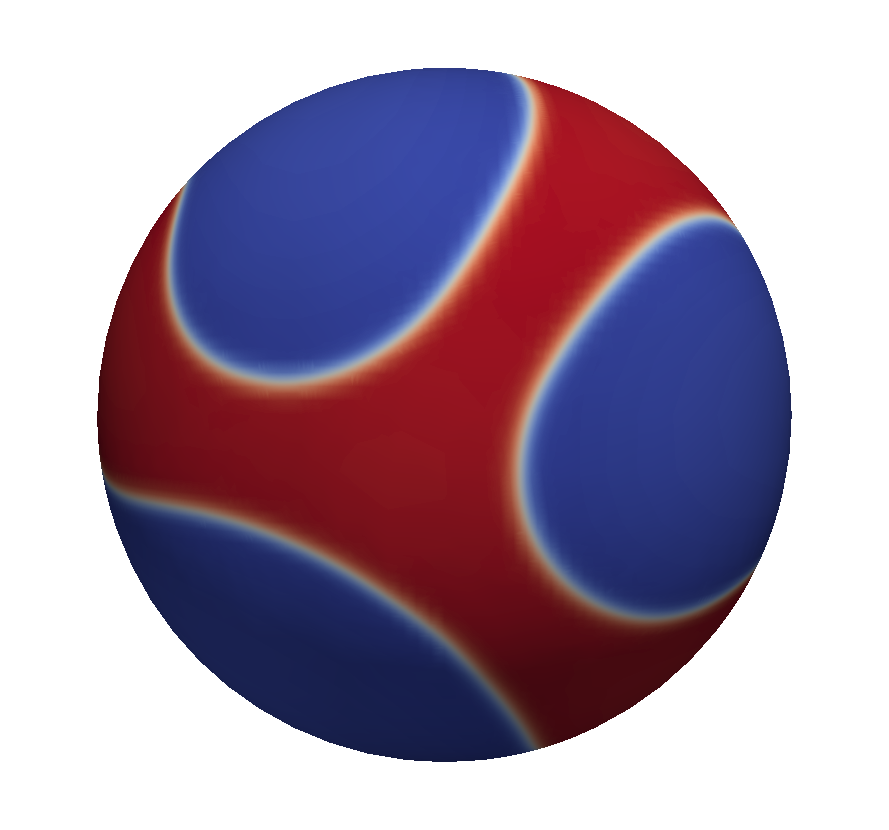}
		\caption{$\phi_h(\cdot,t=80)$}
	\end{subfigure}
	\begin{subfigure}{.24\linewidth}
		\centering
		\includegraphics[width=\linewidth]{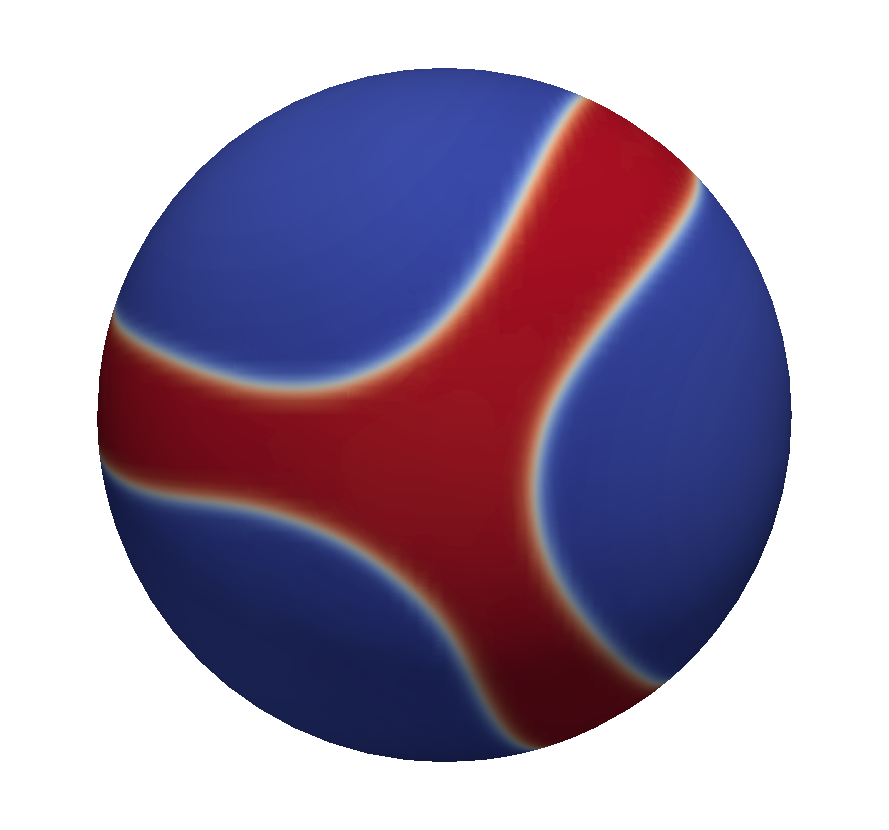}
		\caption{$\phi_h(\cdot,t=200)$}
	\end{subfigure}
	\caption{Unstable state which transitions from 1 domain towards 4 domains. Here for visualiation purposes we don't apply the deformation $u$.}
	\label{fig-instability}
\end{figure}
\subsection{Width of interface, $\epsilon$}
Since we approximated the line tension by the Ginzburg-Landau energy functional, we wish to check that in the limit $\epsilon\to 0$ we see a tightening on the width of the diffuse interface. This is confirmed in Figure \ref{fig:epsilon} where the initial condition was chosen to have icosahedral rotational symmetry, and parameter values $\kappa=1, R=1, b=1, \Lambda=5$ and $\sigma=1$ were used.
\begin{figure}
	\centering
	\begin{subfigure}{.24\linewidth}
		\centering
		\includegraphics[width=\linewidth]{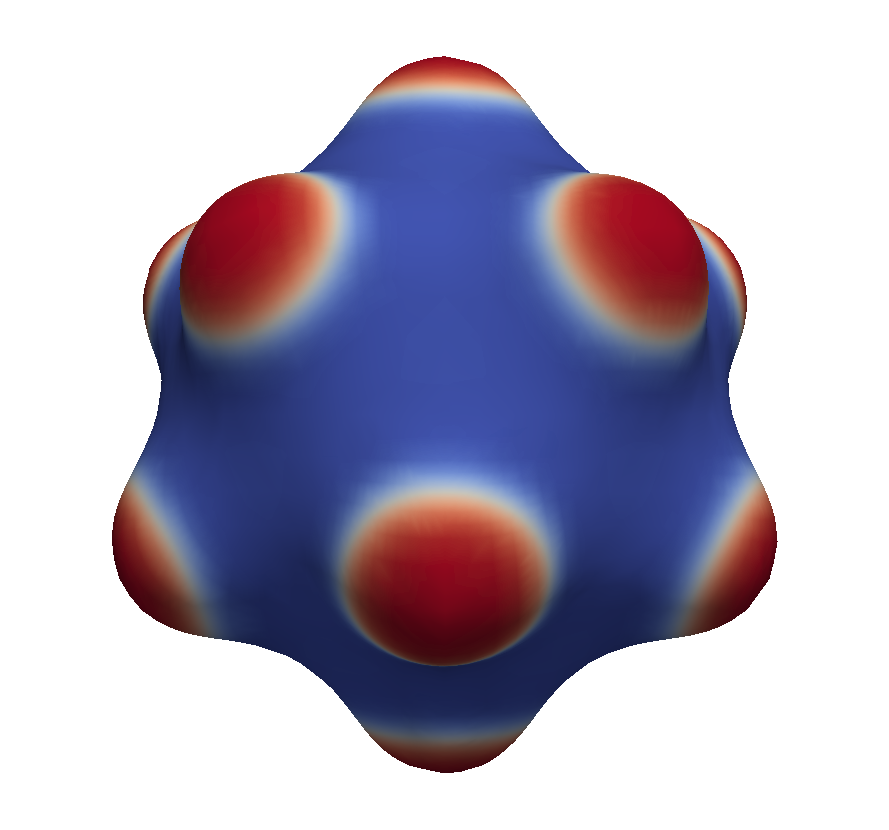}
		\caption{$\epsilon=0.04$}
	\end{subfigure}
	\begin{subfigure}{.24\linewidth}
		\centering
		\includegraphics[width=\linewidth]{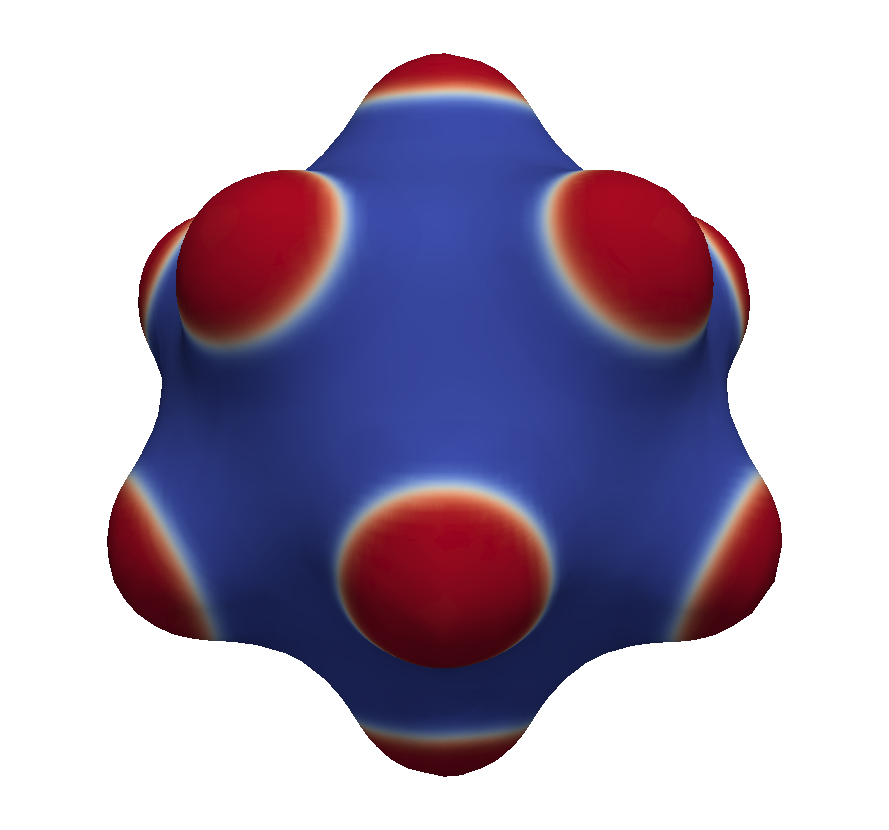}
		\caption{$\epsilon=0.02$}
	\end{subfigure}
	\begin{subfigure}{.24\linewidth}
		\centering
		\includegraphics[width=\linewidth]{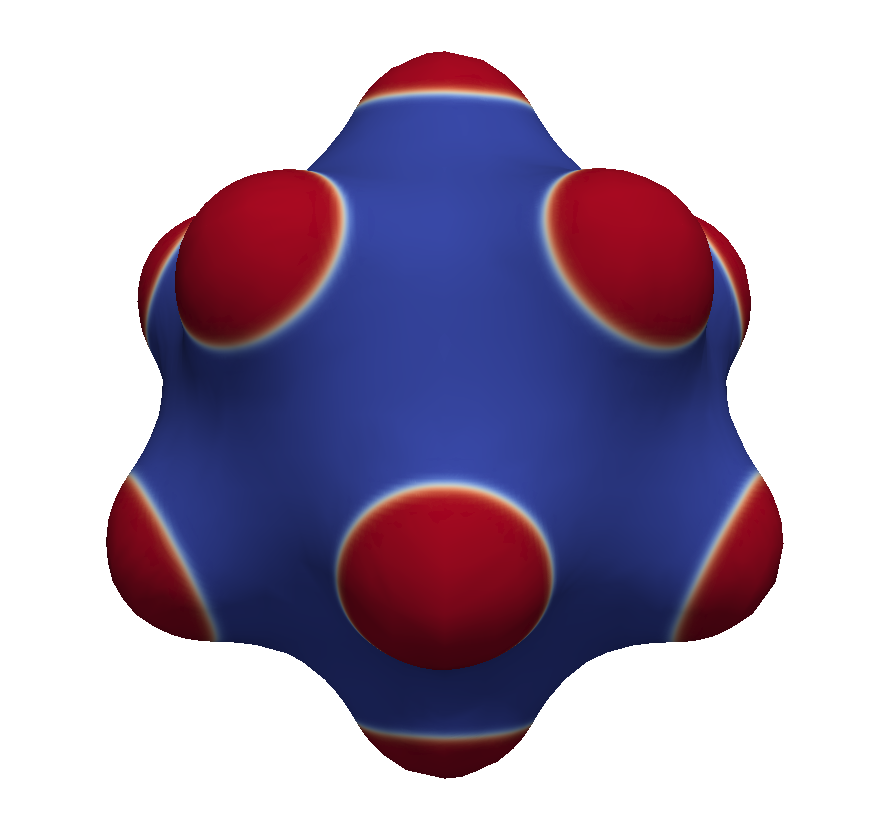}
		\caption{$\epsilon=0.01$}
	\end{subfigure}
	\begin{subfigure}{.24\linewidth}
		\centering
		\includegraphics[width=\linewidth]{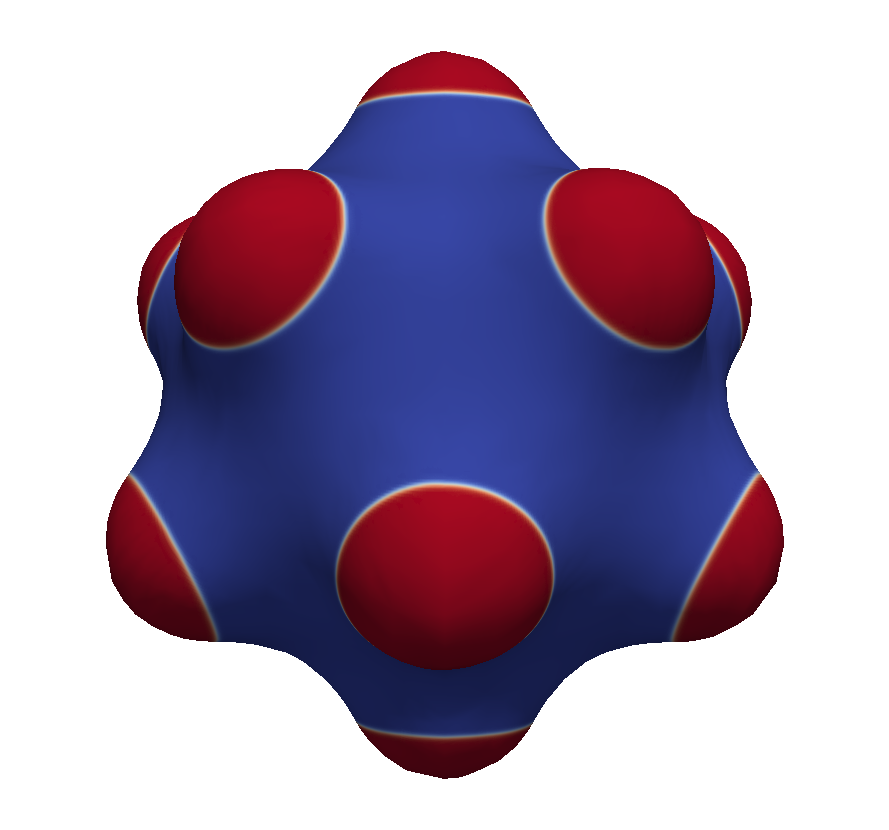}
		\caption{$\epsilon=0.005$}
	\end{subfigure}
	\caption{Almost stationary discrete solutions for varying the width of the interface, $\epsilon$.}
	\label{fig:epsilon}
\end{figure}


 
\subsection{Long time behaviour}
Starting with an initial condition of the form $\phi(\cdot,t=0)=\alpha+\mathcal{R}$ where $\mathcal{R}$ is a given small mean zero random perturbation, we investigate the long time behaviour for varying the different parameters from which a number of interesting geometric features arise.

To start with we set $R=1$ and $\epsilon=0.02$ and consider the parameters $\Lambda=5$, $b=1$, $\alpha=-0.5$, $\sigma=1$ and $\kappa=1$ as a base case, and vary each parameter in turn. Figure \ref{fig:time-evolution} gives a series of snapshots of how the solution varies in time towards an almost stationary state solution, in this case consisting of 12 lipid rafts. Since we have seen that for the same set of parameters it is possible for differing numbers of lipid rafts to stablise, we can't conclude this is a global minimiser, but is indicative of general trends that can be observed for varying certain parameters, e.g. the number of lipid rafts. 
\begin{figure}
	\centering
	\begin{subfigure}{.24\linewidth}
		\centering
		\includegraphics[width=\linewidth]{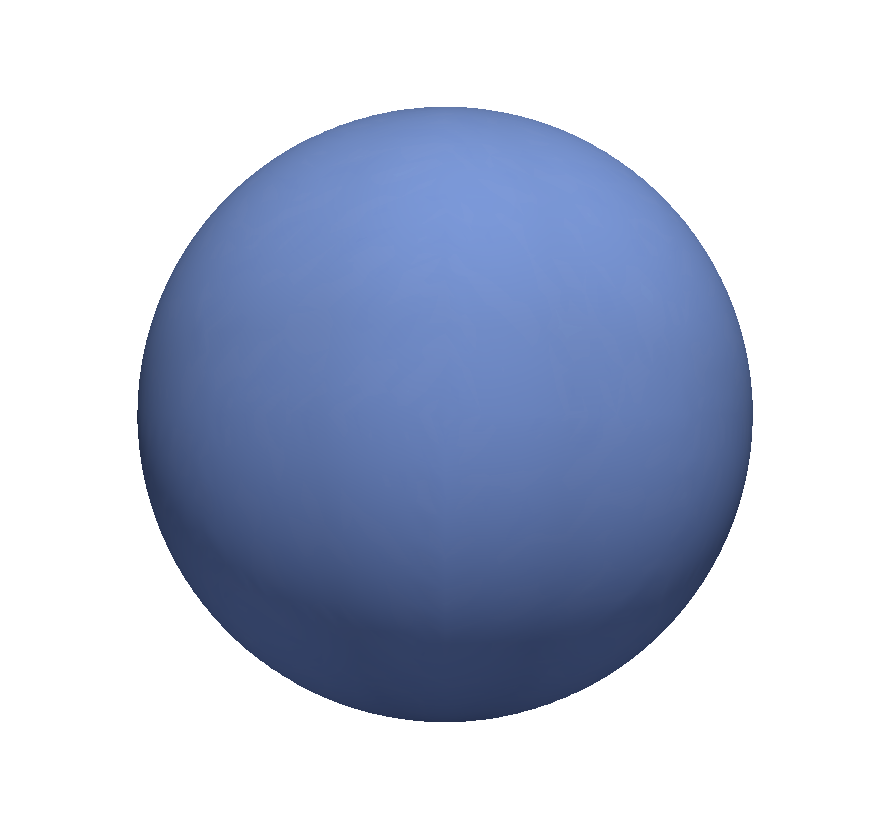}
		\caption{$\phi_h(\cdot,t=0)$}
	\end{subfigure}
	\begin{subfigure}{.24\linewidth}
		\centering
		\includegraphics[width=\linewidth]{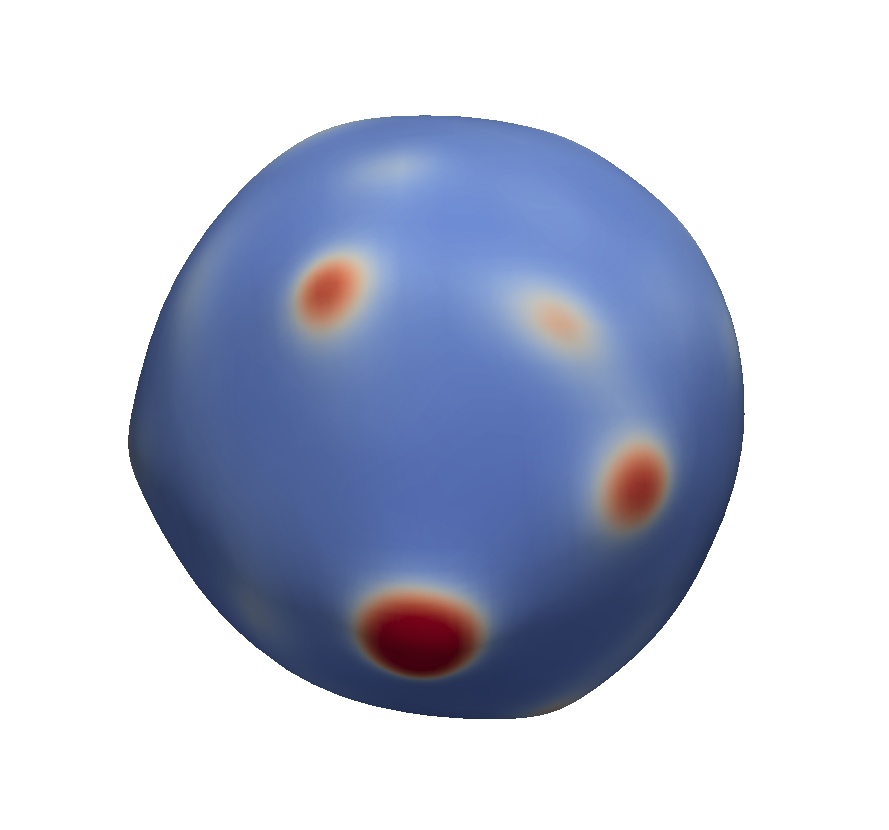}
		\caption{$\phi_h(\cdot,t=0.4)$}
	\end{subfigure}
	\begin{subfigure}{.24\linewidth}
		\centering
		\includegraphics[width=\linewidth]{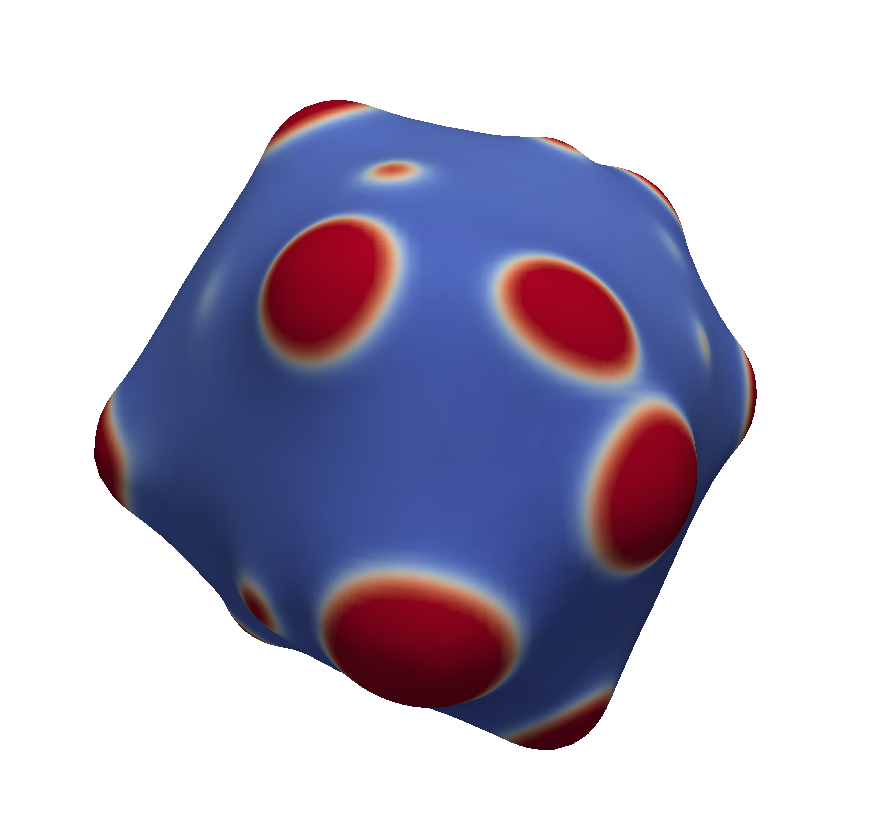}
		\caption{$\phi_h(\cdot,t=0.5)$}
	\end{subfigure}\\
	\begin{subfigure}{.24\linewidth}
		\centering
		\includegraphics[width=\linewidth]{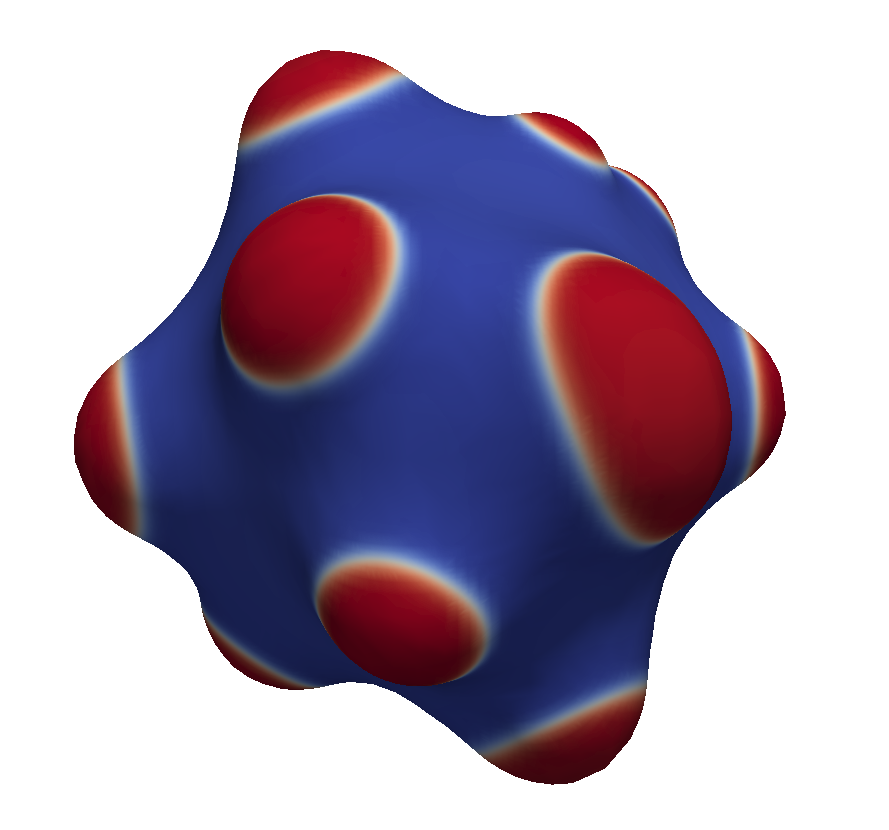}
		\caption{$\phi_h(\cdot,t\approx3.565)$}
	\end{subfigure}
	\begin{subfigure}{.24\linewidth}
		\centering
		\includegraphics[width=\linewidth]{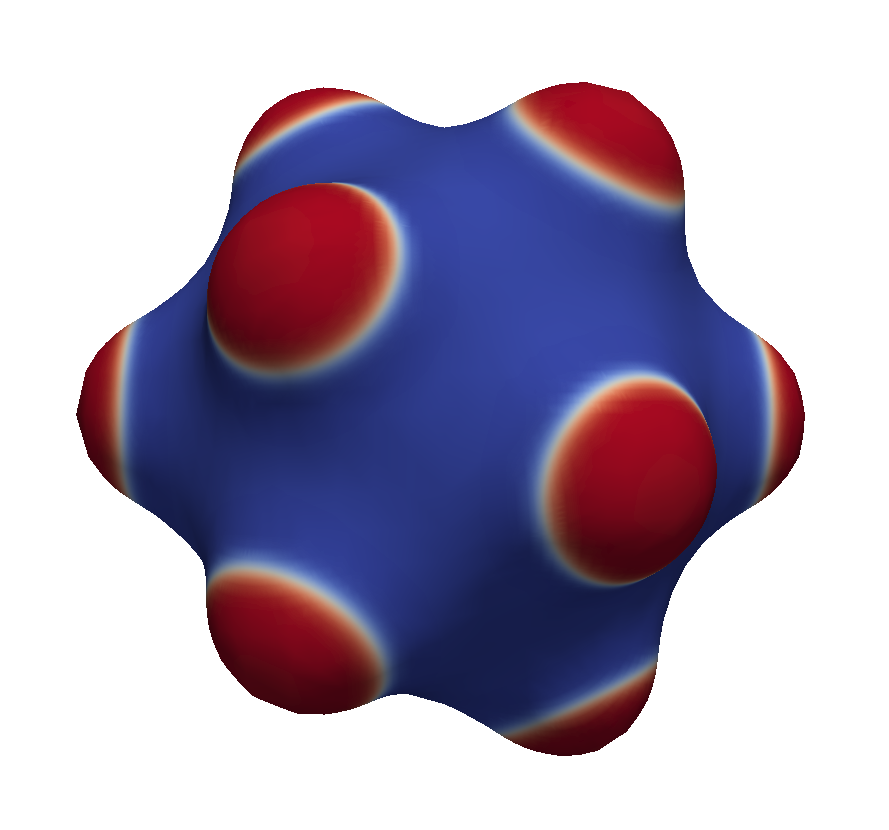}
		\caption{$\phi_h(\cdot,t\approx115.565)$}
	\end{subfigure}
	\begin{subfigure}{.24\linewidth}
		\centering
		\includegraphics[width=\linewidth]{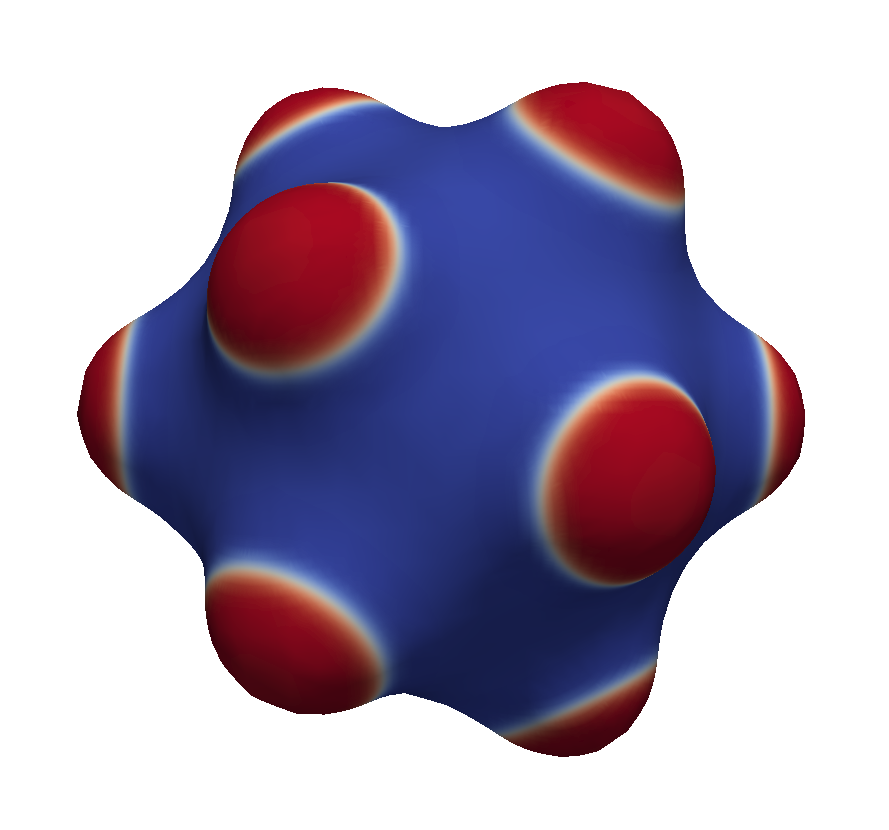}
		\caption{$\phi_h(\cdot,t\approx515.565)$}
	\end{subfigure}
	\caption{The time evolution for initial condition $\phi(\cdot,t=0)=-0.5+\mathcal{R}$ with parameters given by $\Lambda=5$, $b=1$, $\sigma=1$ and $\kappa=1$.}
	\label{fig:time-evolution}
\end{figure}
\subsubsection{Spontaneous curvature, $\Lambda$}
In the case $\Lambda=0$, then there is no coupling so $u=0$ for all time, and $\phi$ evolves according to a conserved Allen-Cahn equation. We observe that as $|\Lambda|$ increases so do the number of lipid rafts, see Figure \ref{fig:Coupling}. This is not surprising since to minimise the energy $\mathcal{E}$, larger $\Lambda$ corresponds to increased curvature. As expected the energy $\mathcal{E}$ coincides for positive and negative values of $\Lambda$ since switching the sign of $\Lambda$  amounts to switching the sign of $u$, which leaves $\mathcal{E}$ unchanged. Further details are given in Table \ref{table:Lambda}.
\begin{figure}
\parbox{.24\linewidth}{\begin{subfigure}{.9\linewidth}
		\centering
		\includegraphics[width=\linewidth]{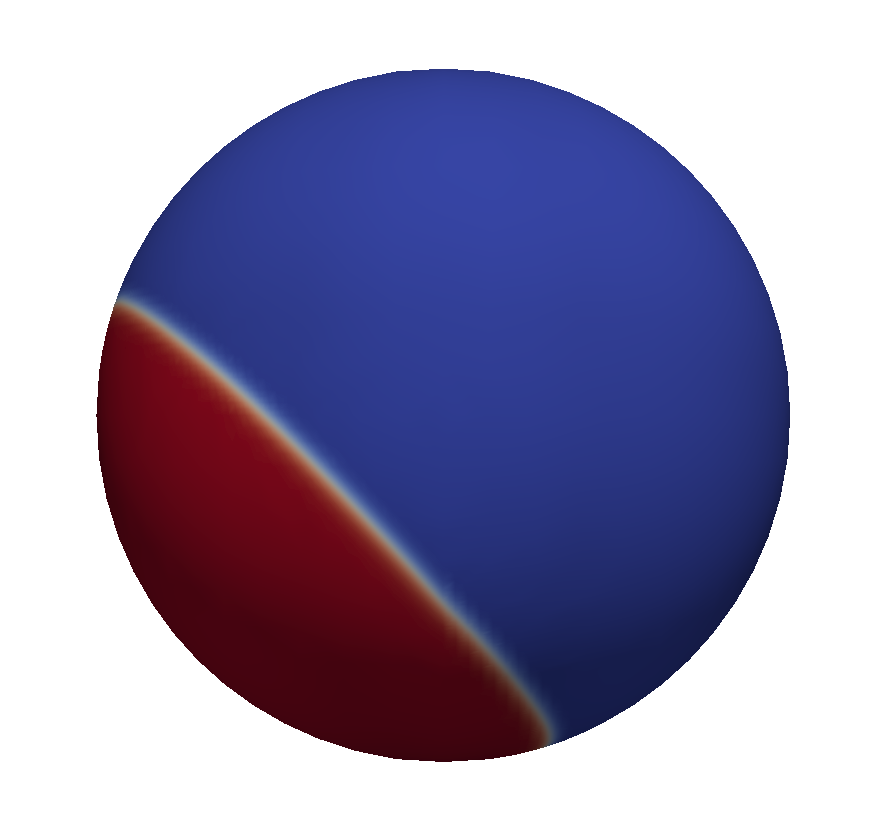}
		\caption{$\Lambda=0$}
	\end{subfigure}}
	\parbox{0.72\linewidth}{
		\begin{subfigure}{.3\linewidth}
			\centering
			\includegraphics[width=\linewidth]{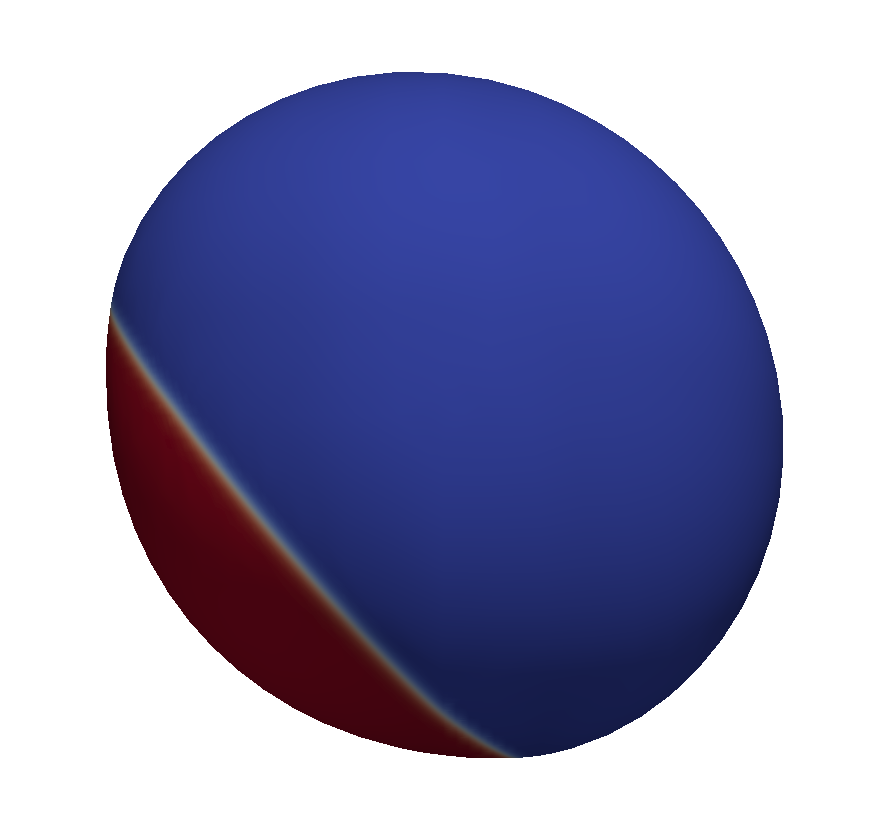}
			\caption{$\Lambda=-0.5$}
		\end{subfigure}
		\begin{subfigure}{.3\linewidth}
			\centering
			\includegraphics[width=\linewidth]{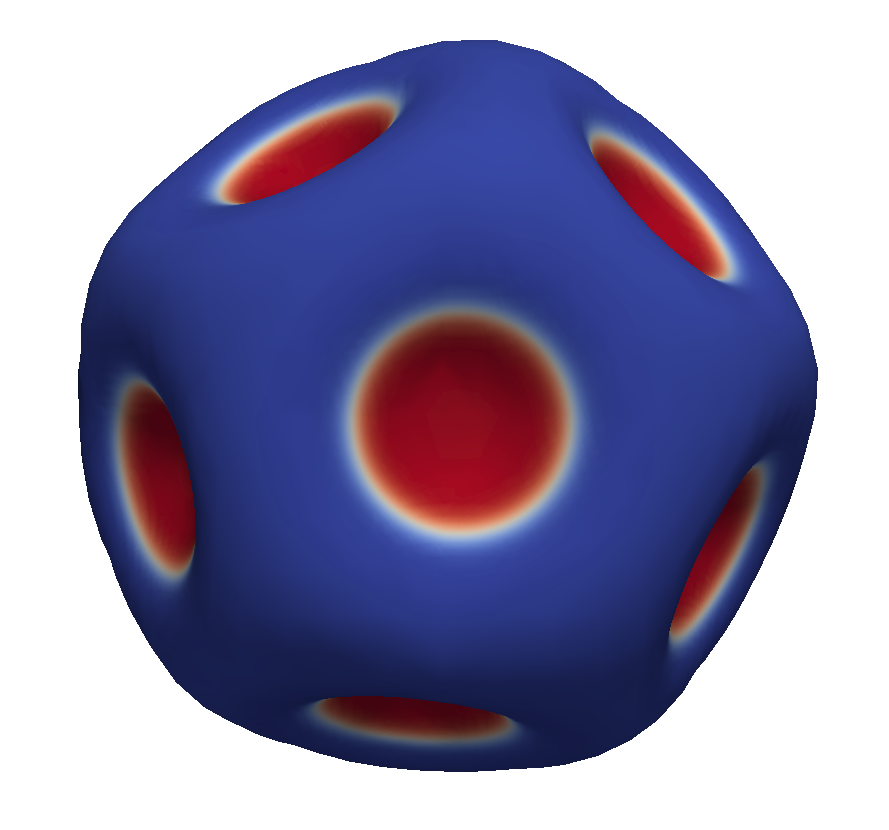}
			\caption{$\Lambda=-5$}
		\end{subfigure}
		\begin{subfigure}{.3\linewidth}
			\centering
			\includegraphics[width=\linewidth]{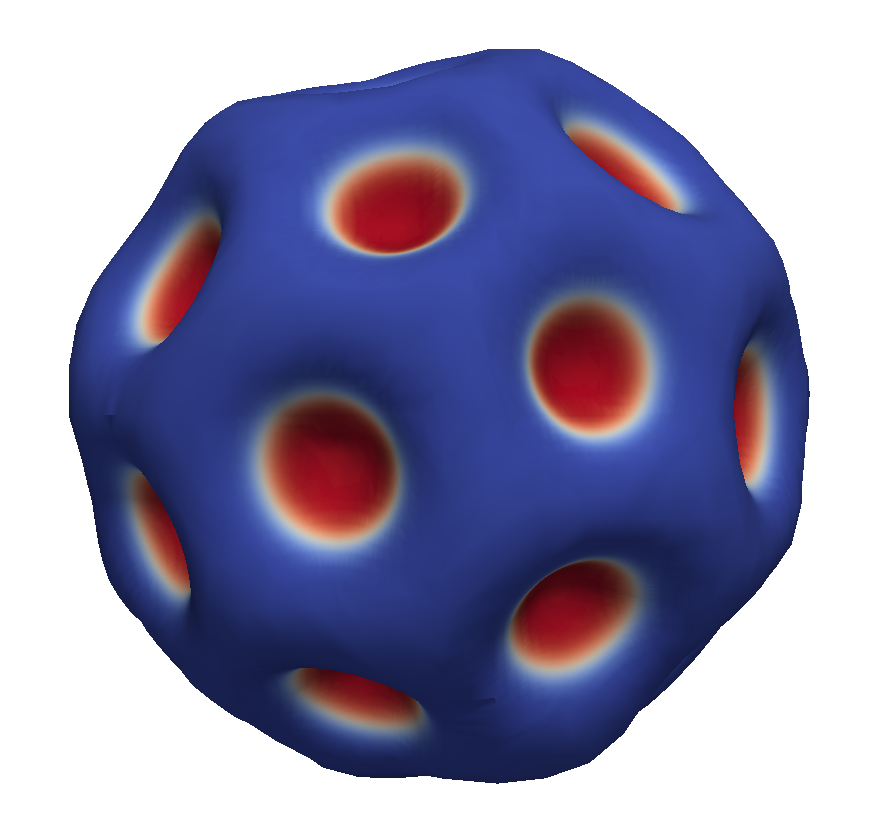}
			\caption{$\Lambda=-10$}
		\end{subfigure}
		\\
		\begin{subfigure}{.3\linewidth}
			\centering
			\includegraphics[width=\linewidth]{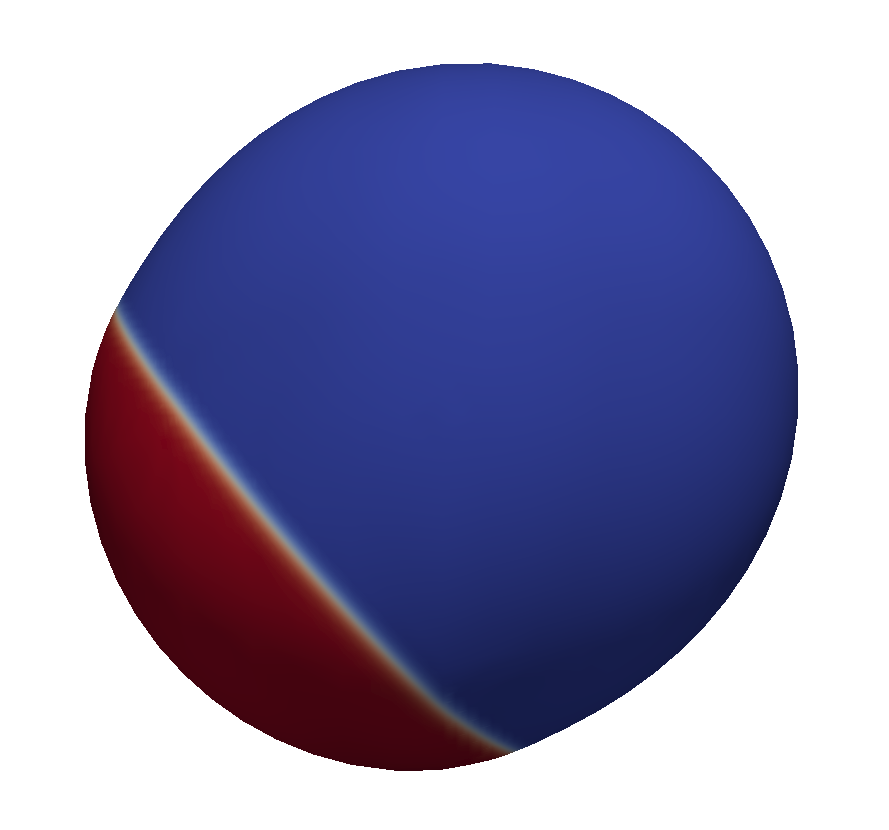}
			\caption{$\Lambda=0.5$}
		\end{subfigure}
		\begin{subfigure}{.3\linewidth}
			\centering
			\includegraphics[width=\linewidth]{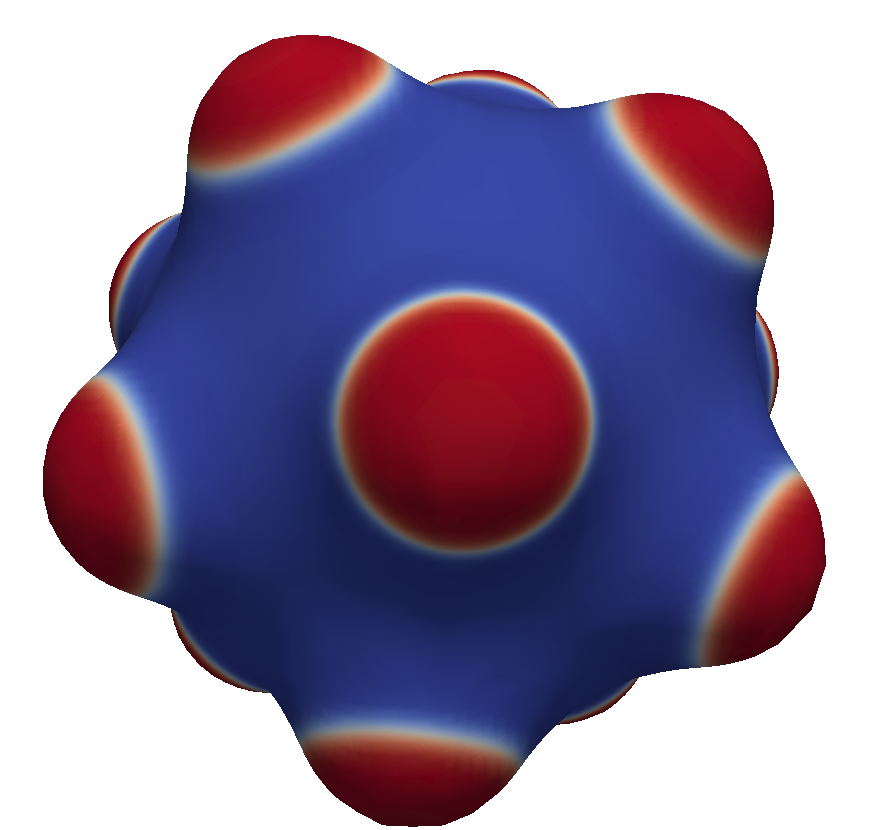}
			\caption{$\Lambda=5$}
		\end{subfigure}
		\begin{subfigure}{.3\linewidth}
			\centering
			\includegraphics[width=\linewidth]{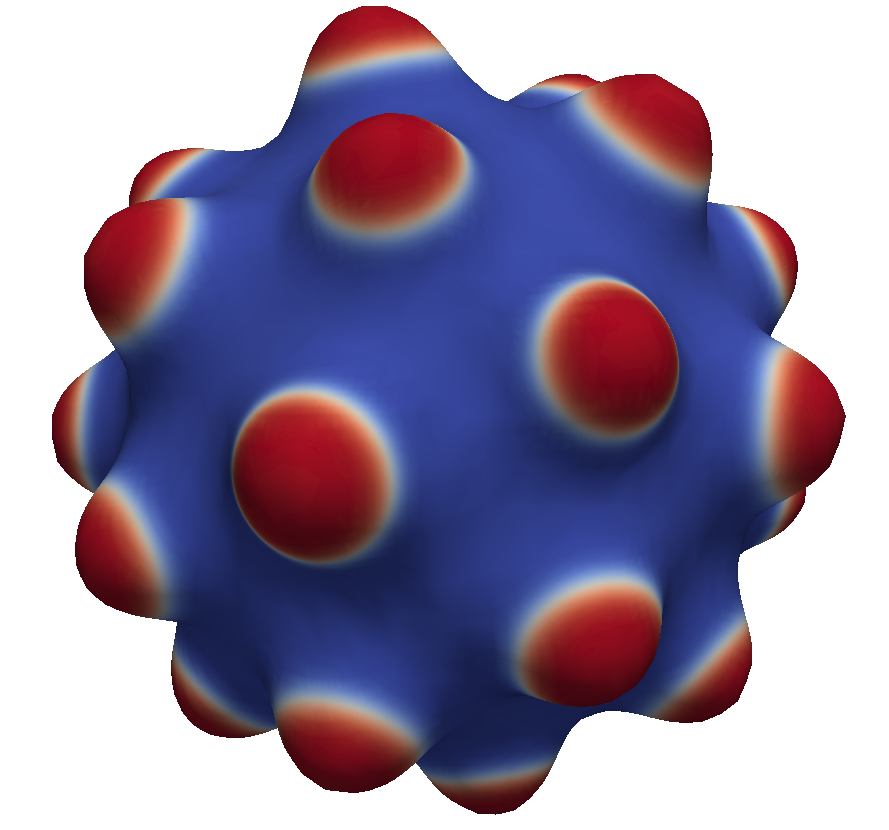}
			\caption{$\Lambda=10$}
		\end{subfigure}}
\centering
\caption{Almost stationary discrete solutions for varying the coupling coefficient $\Lambda$.}
\label{fig:Coupling}\end{figure}
\begin{table}
	\caption{} \label{table:Lambda}
	\begin{minipage}{\textwidth}
		\tabcolsep=8pt
		\begin{tabular}{cccc}
			\hline\hline
			{Figure \ref{fig:Coupling}} & {$\Lambda$} & {\# of lipid rafts}
			& {$\mathcal{E}_h$} \\
			\hline
			(a) & 0	& 1 & 5.1910 \\
			(b) & -0.5  & 1 & 6.3583 \\
			(c) & -5 	& 12 & 66.9928 \\
			(d) & -10  & 26 & 204.9876 \\ 
			(e) & 0.5  & 1 & 6.3583 \\ 
			(f) & 5 	& 12 & 66.9928 \\
			(g) & 10  & 26 & 204.9876 \\
			\hline\hline
		\end{tabular}
	\end{minipage}
\end{table}
\subsubsection{Line tension, $b$}
Similarly, we would expect that increasing the line tension $b$ would decrease the length of the interface, and hence decrease the number of lipid rafts. This agrees with the observed behaviours illustrated in Figure \ref{fig:b}. Further details are given in Table \ref{table:b}. In Figure  \ref{fig:b} (d) we observe that $u$ is positive both in the lipid raft domain but also antipodal to this. This slightly strange behaviour arises from the fact that we have removed the components of the normal $\nu_i$ from $u$. 
\begin{figure}
	\centering
	\begin{subfigure}{.24\linewidth}
		\centering
		\includegraphics[width=\linewidth]{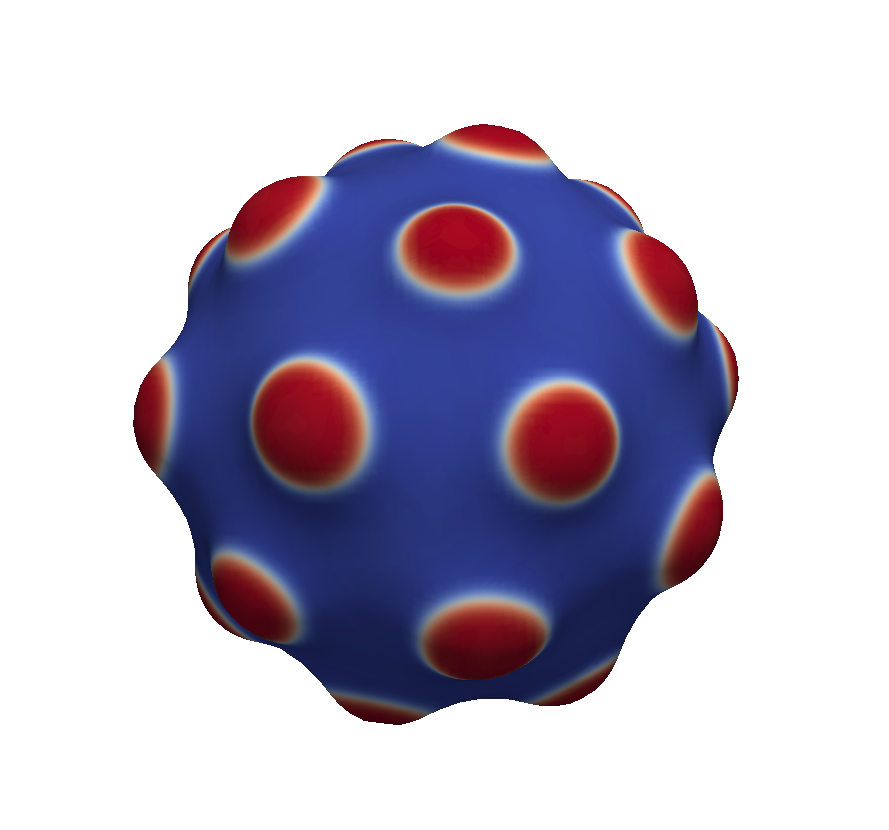}
		\caption{$b=0.2$}
	\end{subfigure}
	\begin{subfigure}{.24\linewidth}
		\centering
		\includegraphics[width=\linewidth]{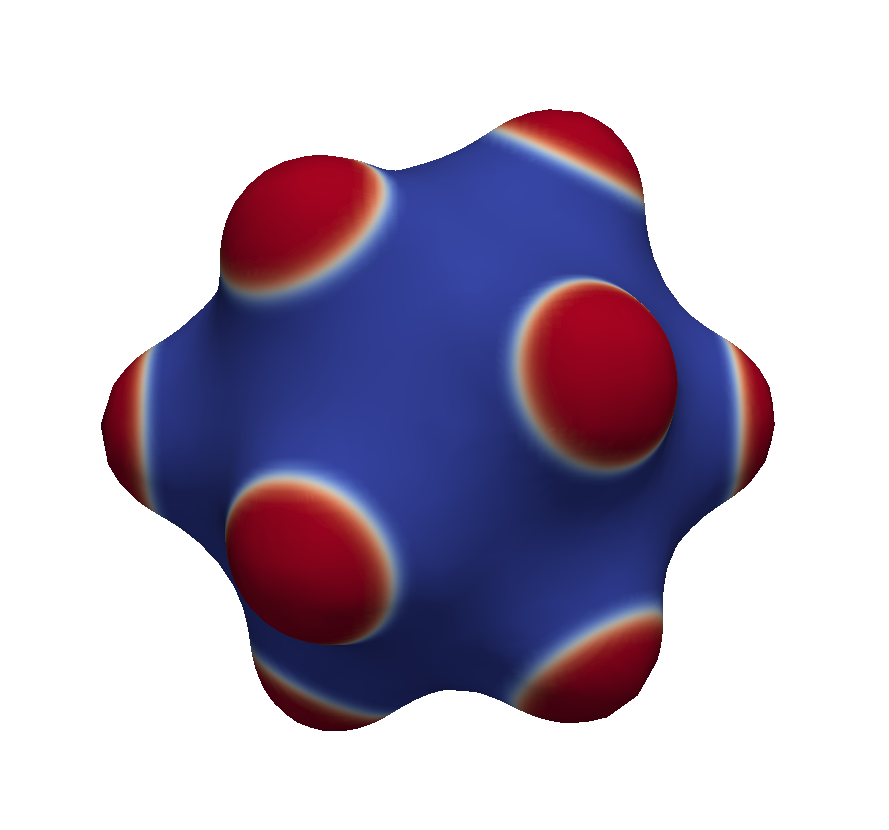}
		\caption{$b=1$}
	\end{subfigure}
	\begin{subfigure}{.24\linewidth}
		\centering
		\includegraphics[width=\linewidth]{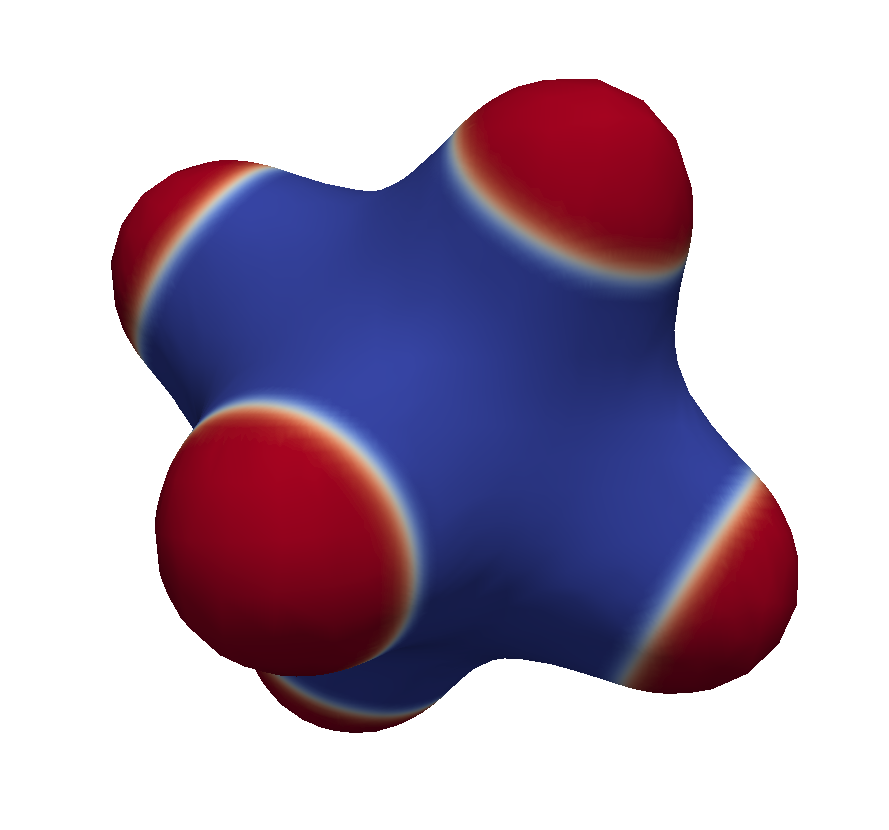}
		\caption{$b=2.5$}
	\end{subfigure}
	\begin{subfigure}{.24\linewidth}
		\centering
		\includegraphics[width=\linewidth]{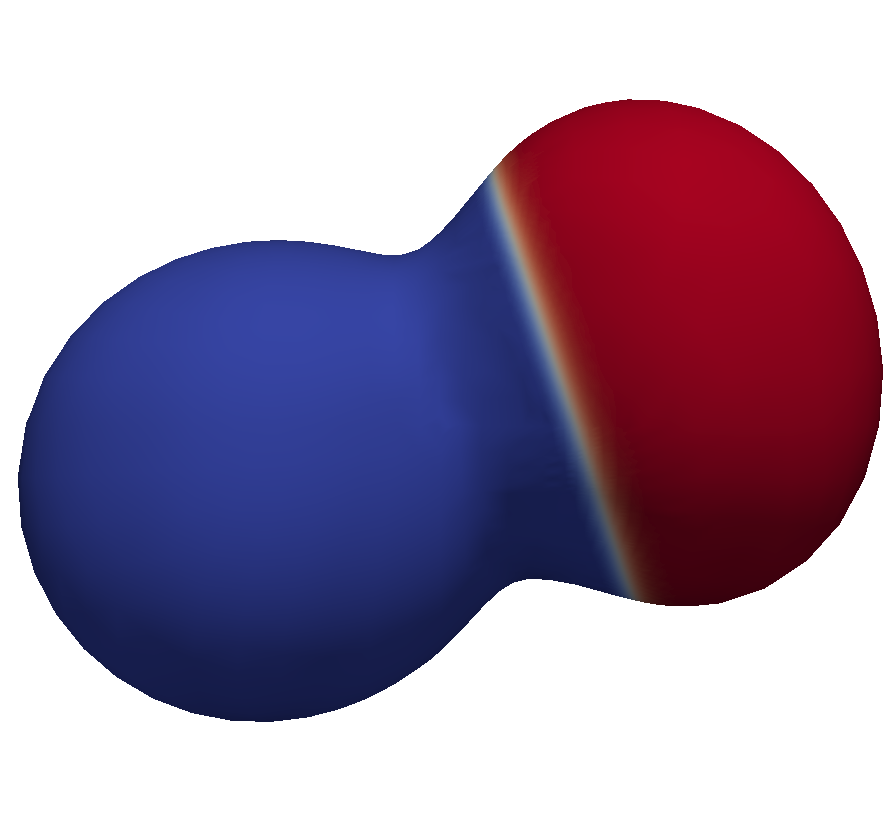}
		\caption{$b=50$}
	\end{subfigure}
	\caption{Almost stationary discrete solutions for varying the line tension term $b$.}
	\label{fig:b}
\end{figure}
\begin{table}
	\caption{} \label{table:b}
	\begin{minipage}{\textwidth}
		\tabcolsep=8pt
		\begin{tabular}{cccc}
			\hline\hline
			{Figure \ref{fig:b}} & {$b$} & {\# of lipid rafts}
			& {$\mathcal{E}_h$} \\
			\hline
			(a) & 0.2	& 26 & 49.9075 \\
			(b) & 1  & 12 & 66.9928 \\
			(c) & 2.5 	& 6 & 88.0572 \\
			(d) & 50  & 1 & 376.0834 \\
			\hline\hline
		\end{tabular}
	\end{minipage}
\end{table}
\subsubsection{Mean value of $\phi$}
Figure \ref{fig:alpha} shows the effect of varying the mean value of $\phi$, with both stripe and circular raft behaviour observed, as well as no phase separation. Further details are given in  Table \ref{table:alpha}. Although Figure \ref{fig:alpha} (a) is almost stationary, its non-symmetric nature is suggestive that this is not a local minimiser.
\begin{figure}
\centering
\begin{subfigure}{.24\linewidth}
	\centering
	\includegraphics[width=\linewidth]{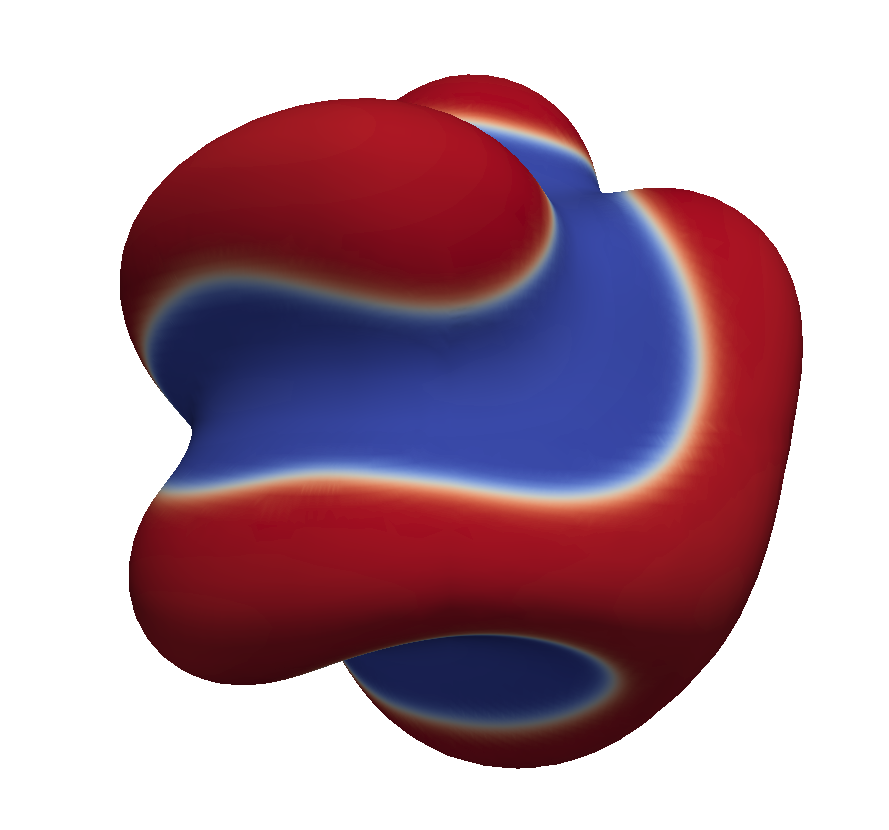}
	\caption{$\alpha=0$}
\end{subfigure}
\begin{subfigure}{.24\linewidth}
	\centering
	\includegraphics[width=\linewidth]{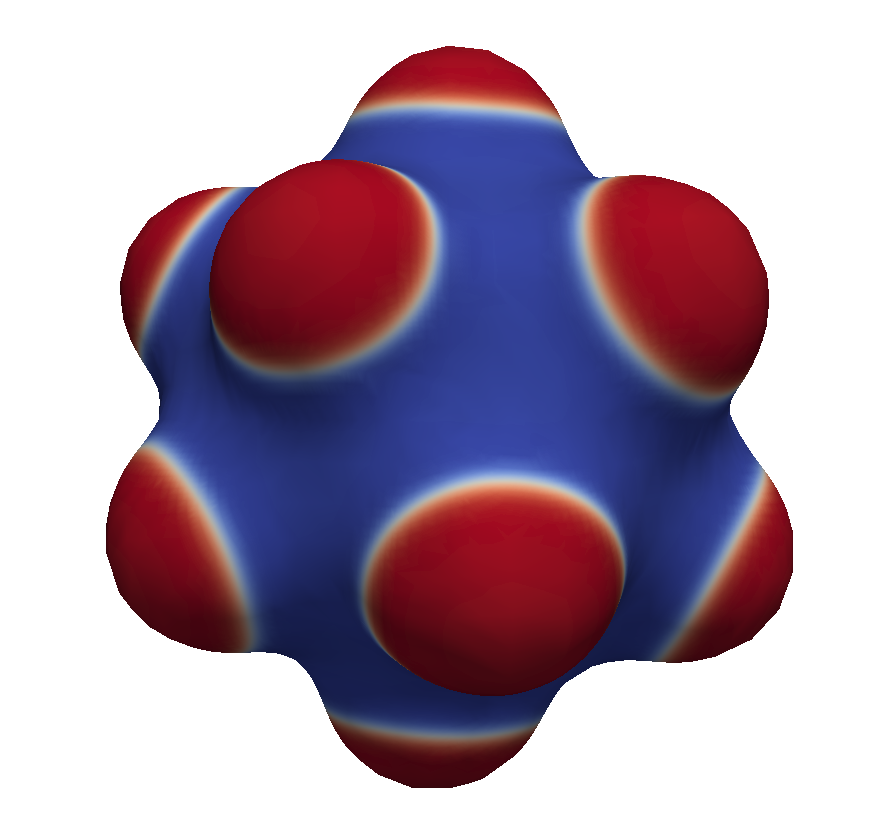}
	\caption{$\alpha=-0.25$}
\end{subfigure}
\begin{subfigure}{.24\linewidth}
	\centering
	\includegraphics[width=\linewidth]{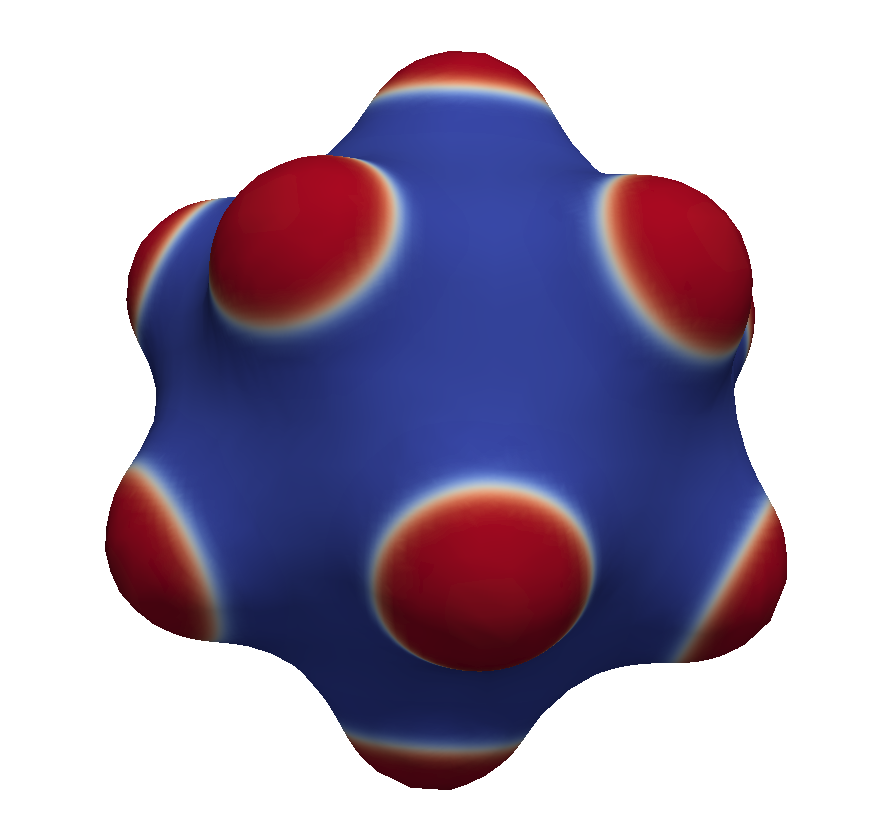}
	\caption{$\alpha=-0.5$}
\end{subfigure}
\begin{subfigure}{.24\linewidth}
	\centering
	\includegraphics[width=\linewidth]{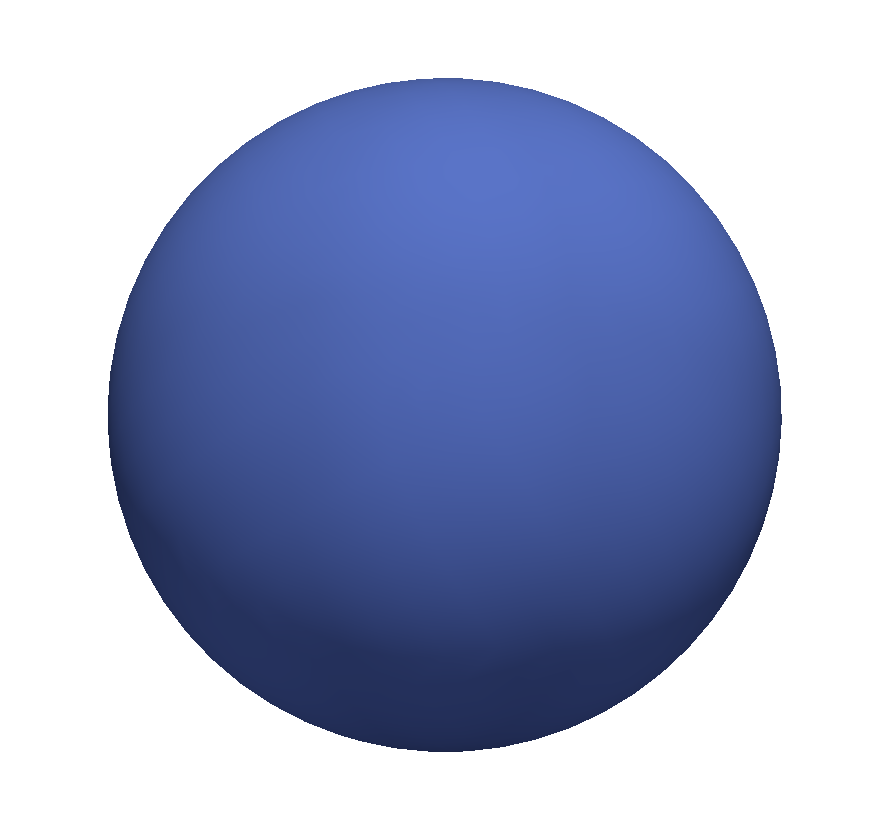}
	\caption{$\alpha=-0.75$}
\end{subfigure}
	\caption{Almost stationary discrete solutions for varying $\alpha$ - the mean value of the order parameter $\phi$.}
	\label{fig:alpha}
\end{figure}
\begin{table}
	\caption{} \label{table:alpha}
	\begin{minipage}{\textwidth}
		\tabcolsep=8pt
		\begin{tabular}{cccc}
			\hline\hline
			{Figure \ref{fig:alpha}} & {$\alpha$} & {\# of lipid rafts}
			& {$\mathcal{E}_h$} \\
			\hline
			(a) & 0	& - & 35.5574 \\
			(b) & -0.25  & 12  & 44.2027 \\
			(c) & -0.5 	&  12& 66.9928 \\
			(d) & -0.75  & - & 118.0643 \\
			\hline\hline
		\end{tabular}
	\end{minipage}
\end{table}
\subsubsection{Surface tension, $\sigma$}
Figure \ref{fig:sigma} shows the effect of varying the surface tension $\sigma$, with increasing $\sigma$ corresponding to increasing numbers of lipid rafts. Further details are given in  Table \ref{table:sigma}. Since in the case $\sigma=0$, there is not a unique solution to \eqref{secant2}, we used a nullspace method from PETSc to enforce that $\int u=0$.
\begin{figure}
	\centering
	\begin{subfigure}{.24\linewidth}
		\centering
		\includegraphics[width=\linewidth]{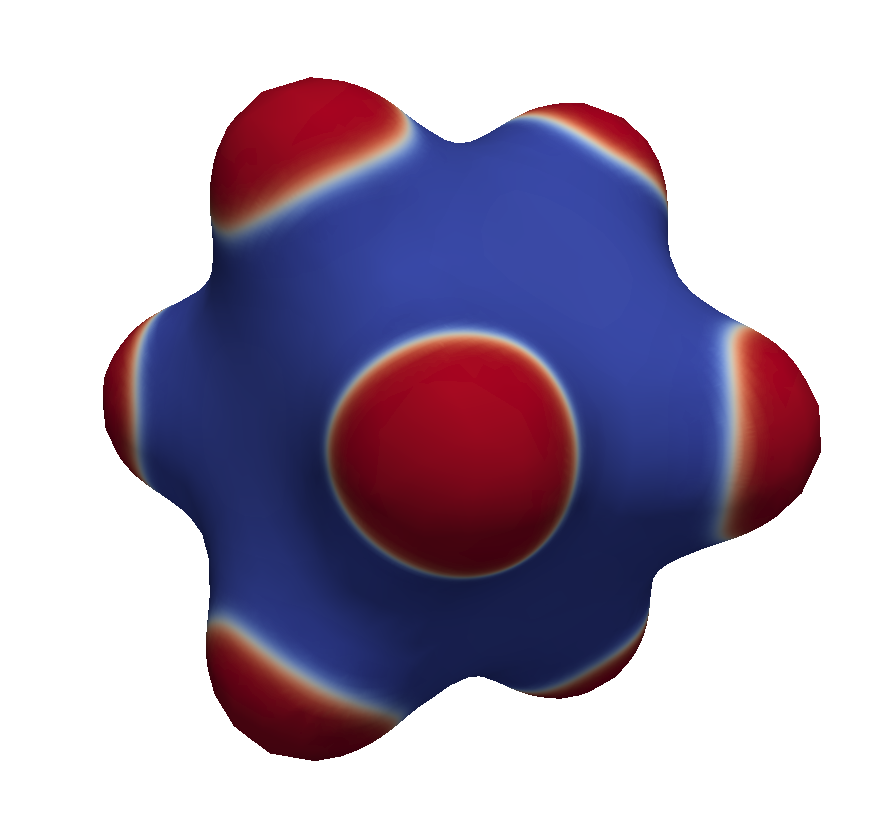}
		\caption{$\sigma=0$}
	\end{subfigure}
	\begin{subfigure}{.24\linewidth}
		\centering
		\includegraphics[width=\linewidth]{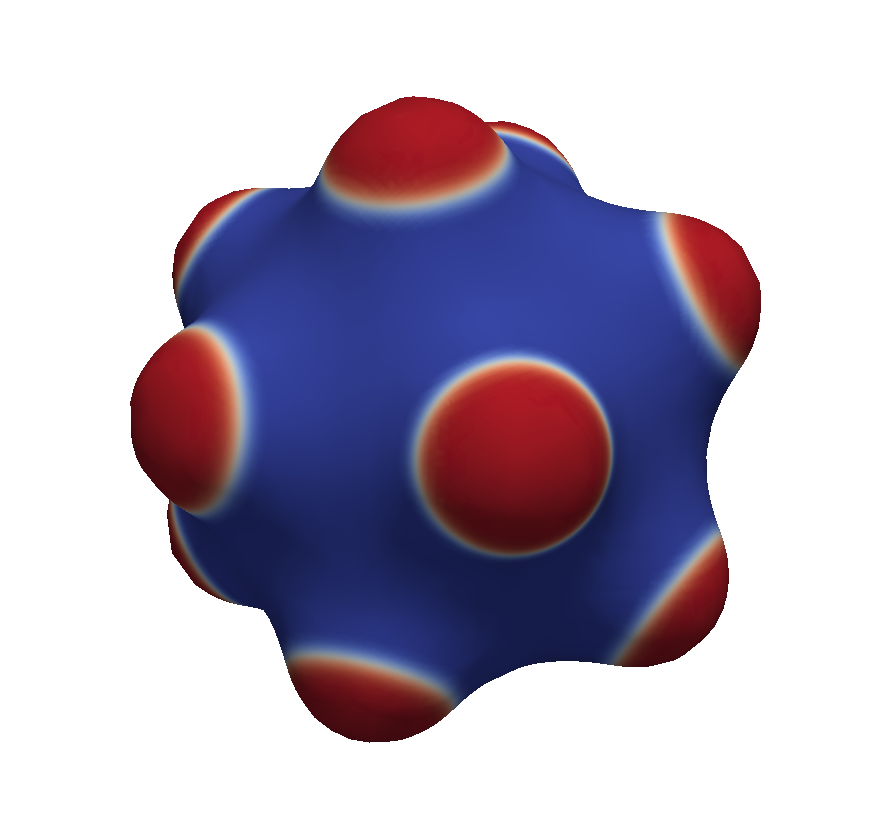}
		\caption{$\sigma=1$}
	\end{subfigure}
	\begin{subfigure}{.24\linewidth}
		\centering
		\includegraphics[width=\linewidth]{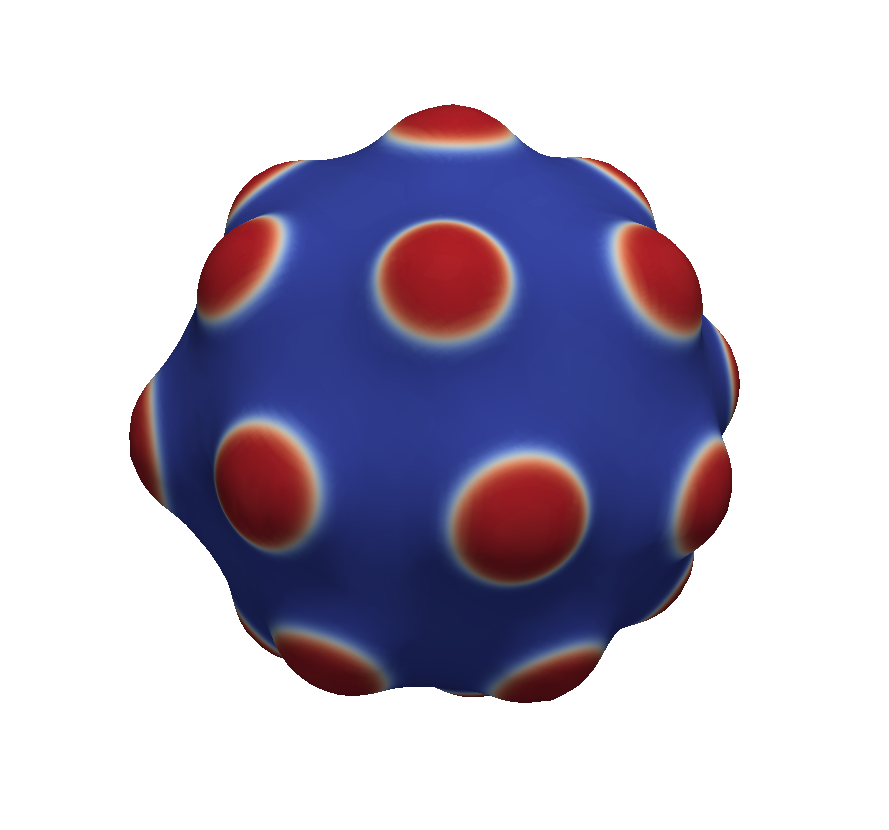}
		\caption{$\sigma=10$}
	\end{subfigure}
	\caption{Almost stationary discrete solutions for varying $\sigma$ - the surface tension.}
	\label{fig:sigma}
\end{figure}
\begin{table}
	\caption{} \label{table:sigma}
	\begin{minipage}{\textwidth}
		\tabcolsep=8pt
		\begin{tabular}{cccc}
			\hline\hline
			{Figure \ref{fig:sigma}} & {$\sigma$} & {\# of lipid rafts}
			& {$\mathcal{E}_h$} \\
			\hline
			(a) & 0	& 8 & 64.0906 \\
			(b) & 1  &12  & 66.9928 \\
			(c) & 10 	& 23 & 79.1846 \\
			\hline\hline
		\end{tabular}
	\end{minipage}
\end{table}
\subsubsection{Bending rigitity, $\kappa$}
Figure \ref{fig:kappa} illustrates the effect of varying the bending rigidity $\kappa$. We observe that increasing $\kappa$ leads to an increase in the number of lipid rafts.  Further details are given in  Table \ref{table:kappa}.
\begin{figure}
\centering
\begin{subfigure}{.24\linewidth}
	\centering
	\includegraphics[width=\linewidth]{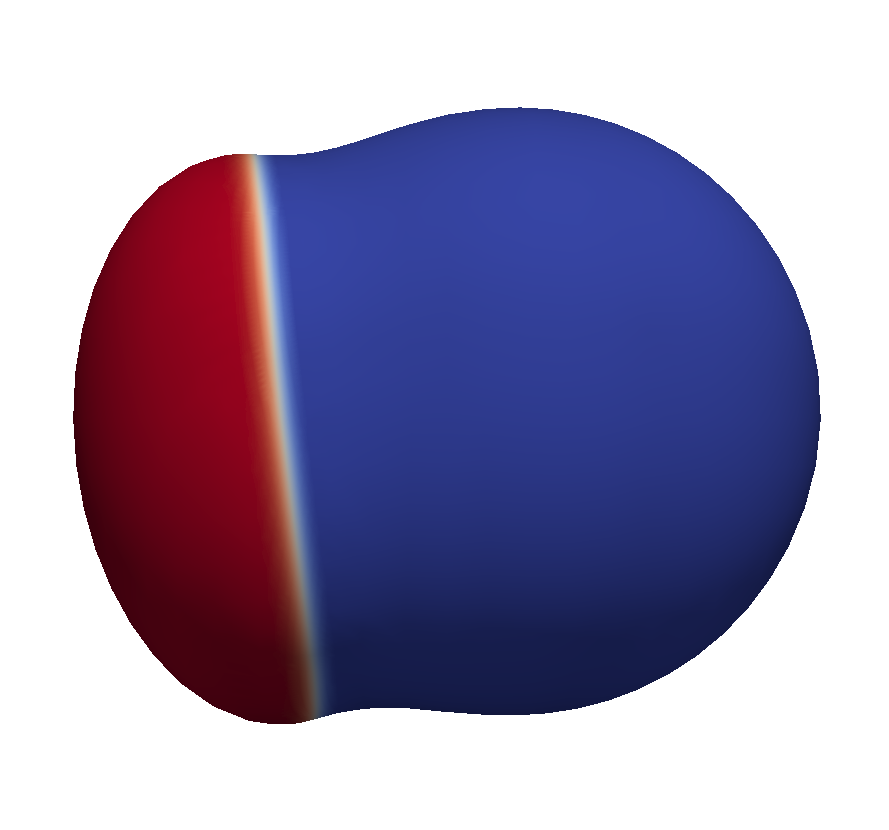}
	\caption{$\kappa=0.05$}
\end{subfigure}
\begin{subfigure}{.24\linewidth}
	\centering
	\includegraphics[width=\linewidth]{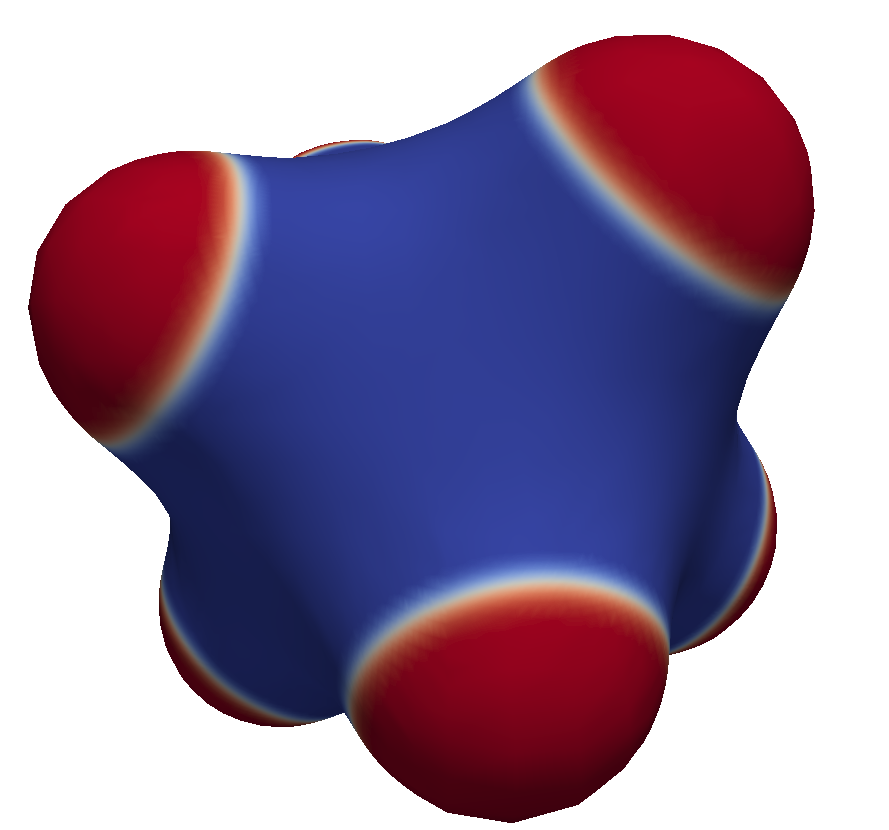}
	\caption{$\kappa=0.1$}
\end{subfigure}
\begin{subfigure}{.24\linewidth}
	\centering
	\includegraphics[width=\linewidth]{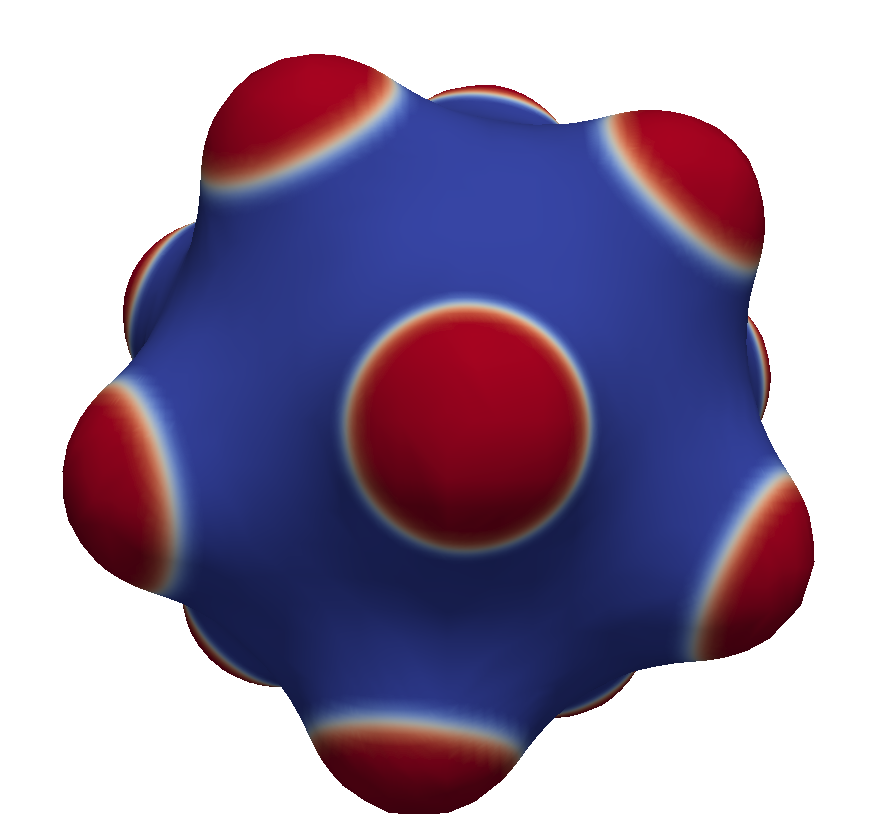}
	\caption{$\kappa=1$}
\end{subfigure}
\begin{subfigure}{.24\linewidth}
	\centering
	\includegraphics[width=\linewidth]{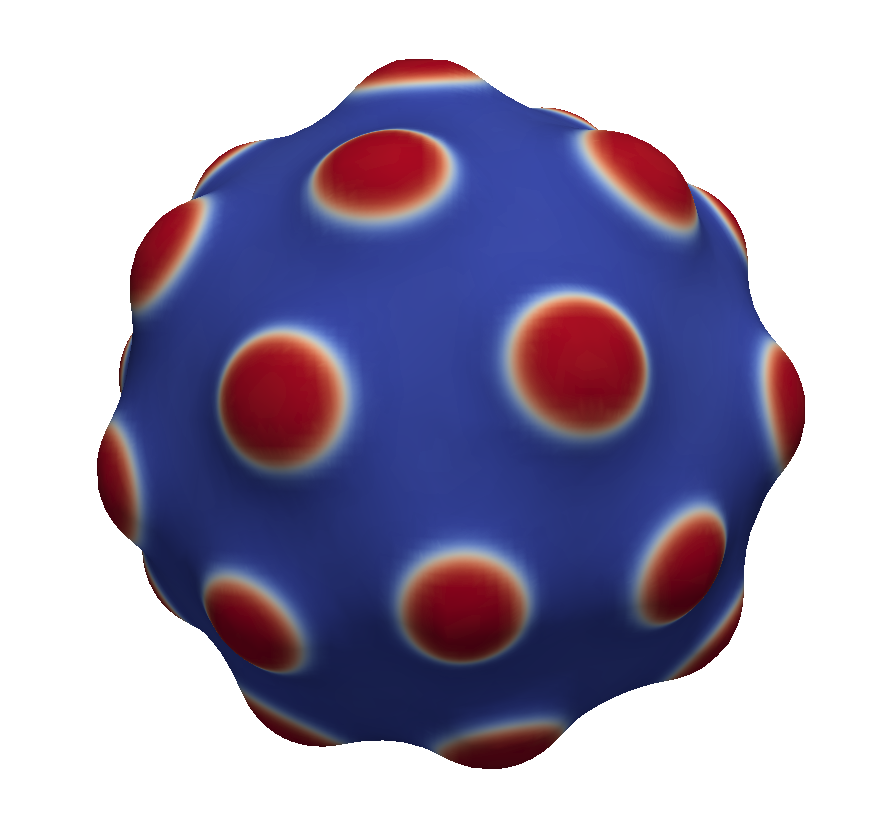}
	\caption{$\kappa=10$}
\end{subfigure}
\caption{Almost stationary discrete solutions for varying the bending rigidity $\kappa$.}
\label{fig:kappa}
\end{figure}
\begin{table}
	\caption{} \label{table:kappa}
	\begin{minipage}{\textwidth}
		\tabcolsep=8pt
		\begin{tabular}{cccc}
			\hline\hline
			{Figure \ref{fig:kappa}} & {$\kappa$} & {\# of lipid rafts}
			& {$\mathcal{E}_h$} \\
			\hline
			(a) & 0.05	& 1 & 16.9889 \\
			(b) & 0.1  & 6 & 37.4941 \\
			(c) &  1	& 12 & 66.9928 \\
			(d) &  10 & 30 & 440.1609 \\
			\hline\hline
		\end{tabular}
	\end{minipage}
\end{table}

\section{Outlook}
The relationship of the diffuse interface approach considered here and a sharp interface problem via asymptotics will be considered in a work in preparation by the authors. Another interesting direction to consider would be a phase-dependent bending rigidity for the Gauss curvature within this perturbation approach, and the exploration of whether this could be sufficient to produce raft like regions as well.
\acknowledgements{
The work of CME was partially supported by the Royal Society via a Wolfson Research Merit Award. 
The research of LH was funded by the Engineering and Physical Sciences Research Council grant  EP/H023364/1
under the MASDOC centre for doctoral training at the University of Warwick.}

\appendix

\label{lastpage}
\end{document}